\documentclass[12pt]{amsart}

\usepackage{tikz}
\usetikzlibrary{snakes}
\usepackage[colorlinks=true, linkcolor=blue, anchorcolor=blue, citecolor=blue, filecolor=blue, menucolor= blue, urlcolor=blue]{hyperref}
\usepackage{amsmath,amssymb,amsthm,amscd}
\usepackage{verbatim}
\usepackage{comment}
\usepackage{multirow}
\usepackage{mathtools}
\usepackage{enumitem}
\usepackage{pgf,tikz}
\usetikzlibrary{arrows}
\usepackage[normalem]{ulem}
\usepackage{graphicx}
\usepackage{amssymb}
\usepackage{wasysym}

\usepackage{arydshln}
\usepackage{faktor}
\usepackage{xfrac}  

\usepackage[margin=1in]{geometry} 
\usepackage{pdfpages}

\numberwithin{equation}{section}

\newtheorem{theorem}{Theorem}[section]
\newtheorem{proposition}[theorem]{Proposition}
\newtheorem{lemma}[theorem]{Lemma}
\newtheorem{corollary}[theorem]{Corollary}

\newtheorem{conjecture}[theorem]{Conjecture}
\newtheorem*{theorem*}{Theorem}

\theoremstyle{definition}
\newtheorem{definition}[theorem]{Definition}
\newtheorem{notation}[theorem]{Notation}
\newtheorem{example}[theorem]{Example}
\newtheorem{remark}[theorem]{Remark}

\newcommand{\ZZ}{ \ensuremath{\mathbb{Z}}}

\newcommand{\PP}{ \ensuremath{\mathbb{P}}}

\newcommand{\gin}{\ensuremath{\mathrm{gin}}}

\newcommand{\adim}{\ensuremath{\mathrm{adim}}}
\newcommand{\edim}{\ensuremath{\mathrm{edim}}}
\newcommand{\vdim}{\ensuremath{\mathrm{vdim}}}

\DeclareMathOperator{\coker}{coker}

\newcommand{\init}{\ensuremath{\mathrm{in}}\hspace{1pt}}

\definecolor{MyDarkGreen}{cmyk}{0.7,0,1,0}

\def\cocoa{{\hbox{C\kern-.13em o\kern-.07em C\kern-.13em o\kern-.15em A}}}

\begin{document}

\title[Expecting the unexpected]{Expecting the unexpected: quantifying the persistence of unexpected hypersurfaces}

\author{G.\ Favacchio}
\address{Dipartimento di Matematica e Informatica\\
Universit\`a di Catania\\
Viale A. Doria, 6\\
I-95125 Catania, Italy}
\email{favacchio@dmi.unict.it}

\author{E.\ Guardo}
\address{Dipartimento di Matematica e Informatica\\
Universit\`a di Catania\\
Viale A. Doria, 6\\
I-95125 Catania, Italy}
\email{guardo@dmi.unict.it}

\author{B.\ Harbourne}
\address{Department of Mathematics\\
University of Nebraska\\
Lincoln, NE 68588-0130 USA}
\email{bharbourne1@unl.edu}

\author{J.\ Migliore} 
\address{Department of Mathematics \\
University of Notre Dame \\
Notre Dame, IN 46556 USA}
\email{migliore.1@nd.edu}

\begin{abstract}
If $X \subset \PP^n$ is a reduced subscheme, we say that $X$ admits an unexpected hypersurface of 
degree $t$ for multiplicity $m$ if the imposition of having multiplicity $m$ at a general point $P$ fails to impose 
the expected number of conditions on the linear system of hypersurfaces of degree $t$ containing $X$. 
Conditions which either guarantee the occurrence of unexpected hypersurfaces, or which ensure that they cannot occur,
are not well understand. 
We introduce new methods for studying unexpectedness, such as the use of generic 
initial ideals and partial elimination ideals to clarify when it can and when it cannot occur.
We also exhibit algebraic and geometric properties of $X$ which
in some cases guarantee and in other cases preclude $X$ having
certain kinds of unexpectedness. In addition,
we formulate a new way of quantifying unexpectedness (our AV sequence),
which allows us detect the extent 
to which unexpectedness persists as $t$ increases but $t-m$ remains constant.
Finally, we study to what extent we can detect unexpectedness from the Hilbert function of X.
\end{abstract}

\date{edited: January 28, 2020; compiled \today}

\thanks{
{\bf Acknowledgements}: Favacchio and Guardo were partially supported by the Universit\`a degli Studi di
Catania, ``Piano della Ricerca 2016/2018 Linea di intervento 2" and  by the ``National 
Group for Algebraic and Geometric Structures, and their Applications" (GNSAGA of INdAM).  Harbourne was 
partially supported by Simons Foundation grant \#524858.
Migliore was partially supported by Simons Foundation grant \#309556.
}



\maketitle



\section{Introduction}

A classical kind of problem in algebraic geometry is to consider vanishing conditions on a linear system $\mathcal L$, 
and to ask if the dimension of the resulting linear system is what one would expect based on the dimension of $\mathcal L$ 
and the specific conditions imposed. That is, one asks if the desired vanishing imposes the {\em expected number} of 
conditions on $\mathcal L$. For example, if $\mathcal L$ is the complete linear system of conics in $\PP^2$ 
(which is 5-dimensional) and $P$ is a  point, then vanishing to multiplicity 2 at $P$ imposes three conditions on $\mathcal L$; 
that is, there is a 2-dimensional linear system of conics double at $P$ as expected. However, if we impose vanishing to 
multiplicity 2 at each of two points $P_1$ and $P_2$, we expect $3+3 = 6$ conditions, i.e. we expect there to be no such 
conic, while in fact the double line passing through $P_1$ and $P_2$ is such a conic. Continuing in this direction leads 
to the well-known Segre-Harbourne-Gimigliano-Hirschowitz (SHGH) Conjecture \cite{segre,Harb,Gi,Hi}. 
A conjecture of Laface and Ugaglia addresses the corresponding situation in $\PP^3$ \cite{LU}, while results of 
Alexander and Hirschowitz \cite{AH} partially address the situation in $\PP^n$ for all $n\geq2$, but much remains unknown.

It is not only vanishing conditions imposed by points that is of interest.
Given a general set of $r$ lines in $\PP^n$, one could ask if vanishing on all of these lines with multiplicity 1 
imposes the expected number of conditions on the complete linear system of hypersurfaces of given degree $d$, 
and an affirmative answer was given by Hartshorne and Hirschowitz \cite{HaHi}. Their paper led to much other 
work, such as research in which lines of higher multiplicity are allowed (for just two recent examples, see \cite{DHRST} and \cite{BDSSS}). The work of Hartshorne and Hirschowitz also led to the paper \cite{CCG}, which can be viewed as a direct precursor to the study of unexpected hypersurfaces.

Thus in recent years, a flurry of activity has emerged on this kind of problem. Some of it grew out of a striking 
example in \cite{DIV}, later formalized in \cite{CHMN} and since then branching off in many different directions 
(for some examples, see \cite{CM, DMO, DHRST, FGST, HaHa, HMNT, HMT, S1, S2, Tr}). The path that led to 
this paper began 
by our looking for conditions on a variety $X$ that either automatically force the existence of unexpected hypersurfaces, 
or else force the conclusion that no unexpected hypersurfaces exist. The results we describe below give examples of 
such conditions. But more specifically we can ask: are there conditions on the Hilbert function that force either of 
these outcomes? It turns out that if the Hilbert function forces $X$ to be degenerate in $\PP^n$ then $X$ 
does not admit any unexpected hypersurfaces of any kind in $\PP^n$ (see Corollary \ref{degenerate}), 
regardless of whether it does in the smallest linear space containing it. Beyond that, additional conditions involving the 
geometry of $X$ seem to be involved. Indeed, with a 
minor assumption on the geometry of $X$, there are such Hilbert functions (see Theorem \ref{force unexp}),
but in the setting of non-degenerate, finite sets of points, we conjecture that there 
are no Hilbert functions that force any kind of unexpectedness (see Conjecture \ref{conj about existence}).

We now describe our results in more detail.
Given a subscheme $X$ of $\PP^n_K$, 
its defining saturated homogeneous ideal $I_X\subseteq R=K[\PP^n_K]=K[x_0,\ldots,x_n]$ (where $K$ is a field)
and integers $t\ge m\ge 1$, we define three numbers associated 
to $(X,t,m)$ (see Notation \ref{dim not}). The {\em actual dimension}, ${\adim}(X,t,m)$, is the dimension 
of the vector space of the forms in $I_X$ of degree $t$ vanishing at a general point $P$ with multiplicity $m$. That is,
\[
{\adim}(X,t,m) = \dim [I_X \cap I_P^m]_t.
\]
Next, the {\em virtual dimension}, ${\vdim}(X,t,m)$,  is the dimension of the linear system of the forms of 
degree $t$ in $I_X$  minus the expected number of conditions imposed by taking $P$ with multiplicity $m$. That is,
\[
{\vdim}(X,t,m) = \dim [I_X]_t - \binom{m-1+n}{n}.
\] 
Finally, the {\em expected dimension} ${\edim}(X,t,m)$ is the maximum of ${\vdim}(X,t,m)$ and 0.

Of course, ${\adim}(X,t,m) \ge {\edim}(X,t,m) \geq {\vdim}(X,t,m)$. We say 
that $X$ {\em admits an unexpected hypersurface of degree $t$ vanishing at a general point $P$ with multiplicity $m$} when 
${\adim}(X,t,m)>{\edim}(X,t,m)$, i.e., when
${\adim}(X,t,m)>0$ and ${\adim}(X,t,m)>{\vdim}(X,t,m)$. 

The purpose of this paper is to get a better feel for when unexpected hypersurfaces are forced to occur 
(hence ``expecting the unexpected"). We relate unexpectedness to algebraic and geometric properties (see for example 
Proposition \ref{p.lex segments are bad for unexpectedness} for the former and 
Corollary \ref{degenerate} and Theorem \ref{uct} for the latter)
and we bring to bear methods not previously applied to unexpectedness (such as the use of generic 
initial ideals and partial elimination ideals) to clarify when unexpectedness can and when it cannot be 
expected. When it can, we also study to what extent
it can, which we do by introducing AV sequences 
measuring the gap between the actual dimension and the virtual dimension (see Definition \ref{def of AV}),
which also allows us to frame our work in terms of persistence (i.e., how long does the gap remain
positive?).

In Section \ref{s.Background} we define these AV sequences. 
Specifically, for a given  subscheme $X$ of $\PP^n$ and a non-negative integer $j\ge 0$, we define 
$AV_{X,j}: \ZZ_{>0}\to \ZZ_{\geq0}$ where
$AV_{X,j}(m):={\adim}(X,m+j,m)-{\vdim}(X,m+j,m)$. Studying the difference ${\adim} - {\vdim}$ is 
not new (see for instance \cite{HMT}), but considering it as a sequence is novel. This sequence has interesting 
properties and leads to compact formulas. 
Sometimes we will consider the general case, but even the cases $j=0$ and $j=1$ are interesting 
(see Proposition \ref{AV_X,0(alpha)} and Sections \ref{s.irreducible ACM curve in P3} -- \ref{s.Unmixed curves and unions with finite sets in P3}).

In Section \ref{s. gin and unexp}, for any subscheme $X$ and integer $j\ge 0$, we prove (Theorem \ref{t. AV is an O-sequence})  
that the sequence $AV_{X,j}$ is actually an $O$-sequence, up to a shift. Indeed, it is the 
Hilbert function of the $K$-algebra $R/(\gin(I_X):x^{j+1})$, where $\gin(I_X)$ denotes the 
generic initial ideal with respect to the lexicographic order.
We use this fact to obtain results which ensure the non-existence of unexpected hypersurfaces.  
In particular, if $X$ lies on a hyperplane or if $\gin(I_X)$ is a lex-segment ideal then $X$ does not admit {\em any} unexpected hypersurfaces of any type (Corollary \ref{degenerate} and Proposition \ref{p.lex segments are bad for unexpectedness}), and hence 
${\adim}(X,t,m)$ {\it always} has the expected value.

The fact that the  AV sequence is actually an $O$-sequence raises many other related questions. 
In Section \ref{s.irreducible ACM curve in P3} we use a cohomological interpretation of the AV sequence, 
described in Remark \ref{cohom interp}, to prove that if $X$ is an irreducible, arithmetically Cohen-Macaulay 
(ACM) curve in $\PP^3$, then $AV_{X,1}(m)$ is unimodal and the first part is differentiable 
(see Theorem \ref{part of conj}). Moreover, we have a great deal of experimental evidence which
suggests that if $X$ is smooth then this sequence is also finite and symmetric (see Conjecture \ref{SI conj}).

In Section \ref{ci section} we apply a new method to study the question of unexpected hypersurfaces, 
namely the theory of {\em partial elimination ideals} introduced by Green \cite{Gr}. We consider the 
case of a general codimension 2 complete intersection $C\subset \PP^n$ and we show the 
non-vanishing of $[I_C \cap I_P^m]_t$ for prescribed values of $t$ and $m$ (Proposition \ref{p. codimension 2 complete intersection}). 
We apply this to the case of $n=3$ to show for $t = (a-1)(b-1) + 1$ and $m = (a-1)(b-1)$, that a 
general complete intersection curve $C$ of type $(a,b)$ with 
$2<a\leq b$, admits an unexpected hypersurface of 
degree $t$ for multiplicity $m$ (Proposition \ref{unexpected CI}).

In Section \ref{s.Unmixed curves and unions with finite sets in P3}, in the case of 
either a reduced equidimensional curve  $C$, or the  disjoint union of  a finite set of points and a curve $C$, 
we describe how the $AV$-sequence depends only on geometric information (see Theorems \ref{uct} and \ref{t.C U X}). 

In Section \ref{s.Finite sets of points in P3} we study how knowledge of the Hilbert function, together 
with certain geometric assumptions, can provide information about  unexpected hypersurfaces. We restrict our 
attention to the case of subvarieties of $\PP^3$, but we expect much more can be said in the 
general situation. We use results from \cite{BGM}, where maximal growth of the $h$-vector of a set 
of points $X$ forces the existence of a suitable curve $C$ in the base locus of some component of $I_X$. 
Then we apply some of the results of Section \ref{s.Unmixed curves and unions with finite sets in P3} to 
produce Theorem \ref{force unexp}, already described above. We include several examples to show the 
range of things that can happen for sets with the same Hilbert function.


\section{Background}\label{s.Background} 
We now introduce the main definitions and notation we need, and
begin an investigation of $AV$ sequences. 
We begin by recalling the notion of the $d$-Macaulay representation of a positive integer. The notion goes back to 
Macaulay \cite{Mac} or earlier; see \cite{BH} for an excellent exposition.

\begin{lemma}[\cite{BH} Lemma 4.2.6 and page 161]
Let $d$ be a positive integer. Any $a \in \ZZ_{\geq0}$ can be written uniquely in the form
\[
a = \binom{k_d}{d} + \binom{k_{d-1}}{d-1} + \dots + \binom{k_j}{j}
\]
where $k_d > k_{d-1} > \dots > k_j \geq j \geq 1$ are integers.
\end{lemma}

\noindent Given the above $d$-Macaulay representation of $a$, we define
\[
a^{\langle d \rangle} = \binom{k_d+1}{d+1} + \binom{k_{d-1}+1}{d} + \dots + \binom{k_j+1}{j+1}
\] 
and set $0^{\langle d \rangle} = 0$.

\begin{theorem}[\cite{BH} Theorem 4.2.10] \label{macaulay thm}
Let $K$ be a field and let $h : \ZZ_{\geq0} \rightarrow \ZZ_{\geq0}$ be a numerical function. 
Then $h$ is the Hilbert function of some standard graded $K$-algebra if and only if
\[
h(0) = 1 \ \ \ \hbox{ and } \ \ \ h(d+1) \leq h(d)^{\langle d \rangle}
\]
for all $d \geq 1$.
\end{theorem}

\noindent An infinite sequence $a_0,a_1,a_2,\ldots$ of non-negative integers with $a_i=h(i)$ for an $h$ 
satisfying the conditions of 
Theorem~\ref{macaulay thm} is called an {\em $O$-sequence}. 
We will regard a finite sequence $a_0,a_1,a_2,\ldots, a_r$ as an infinite sequence
by setting $a_i=0$ for $i>r$.

Hereafter, let $R = K[x_0,\dots, x_n]$ be a polynomial ring over a field $K$.
Our default assumption will be that $K$ is algebraically closed,
but sometimes we will need the characteristic to be zero, 
and sometimes we will need only that $K$ is infinite;
in these cases we will say so explicitly.
  
\begin{notation}
For any subvariety (or subscheme) $V \subseteq \PP^n$ we write $I_V \subseteq R$ for the saturated ideal of $V$
and $\mathcal{I}_V$ for the sheaf on $\PP^n$ corresponding to $I_V$. For a standard graded algebra $R/I$ 
we write $h_{R/I}(t)$ for the {\em Hilbert function} of $R/I$, i.e. $h_{R/I}(t) = \dim_K [R/I]_t$. When $I = I_V$ for 
some subscheme $V$, we sometimes write $h_V(t)$ for  $h_{R/I_V}(t)$. We say that $V$ is {\em arithmetically 
Cohen-Macaulay (ACM)} if $R/I_V$ is a Cohen-Macaulay ring.

For any integer function $h : \ZZ_{\geq 0} \to \ZZ$, the first difference $\Delta h$ is the backward difference,
defined by setting $\Delta h(0) = h(0)$ and $\Delta h(t) = h(t)-h(t-1)$ for $t>0$. 
When $X$ is a finite set of points, it is well-known and easy to see that
$h_X$ is strictly increasing until it becomes constant, hence there is a $j$ such that
$1=h_X(0)<\cdots<h_X(j)=h_X(j+1)=\cdots$. 
We refer to $(\Delta h_X(0),\ldots,\Delta h_X(j))$ as 
the {\em $h$-vector} of $X$; it is known to be a finite $O$-sequence
(see, e.g., \cite[Proposition 6.3]{CH}).
\end{notation}


\begin{notation} \label{dim not}
Let $t\ge m$ be positive integers and let $P \in \PP^n$ be a general point. Given $X \subset \PP^n$ a subscheme, 
we set 
\[
\begin{array}{llllll}
{\adim} (X,t,m) \ =\ {\adim} (I_X,t,m) \ =& \displaystyle \dim [I_X \cap I_P^m]_t & \hbox{ (the actual dimension)}, \\
{\vdim} (X,t,m) \ =\ {\vdim} (I_X,t,m) \ =& \displaystyle \dim [I_X]_t - \binom{m+n-1}{n} & \hbox{ (the virtual dimension)},  \\
{\edim} (X,t,m) \ =\ {\edim} (I_X,t,m) \ =& \displaystyle \max \{ 0, {\vdim}(X,t,m) \} & \hbox{ (the expected dimension)}. 
\end{array}
\]
\end{notation}

Note that we always have 
\[
{\adim}(X,t,m) \geq {\edim}(X,t,m) \geq {\vdim}(X,t,m).
\]

\begin{definition}\label{unexpDef}
If ${\adim} (X,t,m) > {\edim}(X,t,m)$, we say that $X$ {\em admits an  unexpected 
hypersurface of degree $t$ for multiplicity $m$}. (In this case, note that ${\adim} (X,t,m)>0$,
hence $t\geq m$.) If $X \subset \PP^2$ is a finite set of 
points which admits an  unexpected hypersurface of degree $t$ for multiplicity $m=t-1$,
then following \cite{CHMN} we say simply that {\em $X$ admits an unexpected curve of degree $t$}.
\end{definition}

\begin{remark} \label{equiv to unexp}
An equivalent condition for $X$ to admit an unexpected hypersurface of degree $t$ for 
multiplicity $m$ is ${\adim}(X,t,m) > 0$ and ${\adim}(X,t,m) > {\vdim}(X,t,m)$.
\end{remark}

\begin{remark}
Any hypersurface of degree $t$ with an isolated singularity of multiplicity $t$ must be a 
cone (by Bezout's theorem). Thus ${\adim}(X,t,t)$ is the dimension of the vector space 
of cones over $X$ of degree $t$ with vertex at~$P$. If ${\adim} (X,t,t) > {\edim}(X,t,t)$,  
we say that $X$ {\em admits an unexpected cone of degree $t$}. See \cite{HMNT, CM, HMT} for 
more on unexpected cones. In particular, if $X$ has codimension two and is reduced, 
equidimensional and non-degenerate then the cone  $S_P$ over $X$ with vertex $P$ is 
an unexpected cone of degree $t=\deg X$ (\cite{HMNT} Proposition 2.4).
\end{remark}

\begin{definition} \label{def of AV}
Let $X \subset \PP^n$ be a closed subscheme. Fixing a non-negative integer $j$, 
we define the sequence $AV_{X,j}$ as follows:
\[
AV_{X,j}(m) = 
\begin{array}{cc}
{\adim}(X,m+j,m) - {\vdim}(X,m+j,m), & m\ge 1.
\end{array}
\] 
\end{definition}

\begin{remark}
Rephrasing Remark \ref{equiv to unexp}, if ${\adim} (X,t,m) > 0$ then $X$ admits 
an unexpected hypersurface of degree $t$ for multiplicity $m$ if and only if $AV_{X,j}(m) > 0$ for $j = t-m$.

\end{remark}

\begin{notation} \label{Pm vs mP}
Let $P \in \PP^n$ be a general point, with defining ideal $I_P$. We will 
denote the scheme defined by $I_P^m$ in $\PP^n$ by $P^m$. We will 
sometimes consider the hyperplane section of $P^m$ by a hyperplane $H$ 
containing $P$, and we will denote the corresponding subscheme of  $H$ by $mP$, thus $mP=P^m\cap H$.
\end{notation}

We now give an interpretation of the sequence $AV_{X,j}$. Notice that, in the 
following lemma, the ideal in the dimension of the quotient on the right changes with $m$. 

\begin{lemma}\label{l.AV seq 1} 
Let $X \subset \PP^n$ be a subscheme.  Then
	\[
	AV_{X,j}(m) = \dim\left[  \faktor{R}{(I_X + I_P^m)}\right]_{m+j}  .
	\] 
\end{lemma}
\begin{proof}
	Set $t:=m+j$. From the short exact sequence
	\[
	0\to R/(I_X \cap I_P^m) \to R/I_X \oplus R/ I_P^m \to R/(I_X + I_P^m)\to 0
	\]
	we get the relation between the dimension of the modules in degree $t$
	\[
	\dim[R]_t- {\adim} (X,t,m) -\dim[R]_t+\dim[I_X]_t - \dim[R]_t+ \dim [I_P^m]_t +\dim[ R/(I_X + I_P^m)]_t= 0.
	\]
	Therefore
	\[
	\begin{array}{rcl}
	\dim[ R/(I_X + I_P^m)]_t&=& {\adim} (X,t,m) -\dim[I_X]_t + \dim[R]_t-\dim [I_P^m]_t =\\
	&=& {\adim} (X,t,m) -\dim[I_X]_t + h_{P^m}(t) .
	\end{array}
	\]
	Since $t\ge m$ the Hilbert function of the fat point $P^m$ in degree $t$ is 
	$h_{P^m}(t)={m+n-1\choose n}$. (It reaches the degree $\deg(P^m)$ in degree $m-1$.) So we get
	\[\dim[ R/(I_X + I_P^m)]_t = {\adim} (X,t,m) - {\vdim} (X,t,m) \] 
	as desired.
\end{proof}

\begin{remark} \label{cohom interp}
	We can also give a cohomological interpretation for the sequence $AV_{X,j}(m)$. 
	 Let $X$ be an ACM subscheme in $\PP^n$ of dimension $\geq 1$. Assume $m \geq 0$ and $j \geq 0$. Consider the exact sequence of sheaves
	\[
	0 \rightarrow \mathcal I_{X \cup P^m} \rightarrow \mathcal I_X \rightarrow \mathcal O_{P^m} \rightarrow 0.
	\]
	Twisting by $m+j$ and taking cohomology gives the exact sequence
	\[
	0 \rightarrow [I_{X \cup P^m}]_{m+j} \rightarrow [I_X]_{m+j} \rightarrow H^0(\mathcal O_{P^m}(m+j)) \rightarrow H^1(\mathcal I_{X \cup P^m}(m+j)) \rightarrow 0.
	\]
	(Exactness on the right is because $X$ is ACM of dimension $\geq 1$; see \cite[pp. 9-11]{migbook}.) This gives
	\[
	h^1(\mathcal I_{X \cup P^m}(m+j)) = \binom{(m-1)+n}{n} - \dim [I_X]_{m+j} + \dim [I_{X \cup P^m}]_{m+j} = AV_{X,j}(m).
	\]
\end{remark}

\begin{remark}
Given a subscheme $X \subset \PP^n$, it is natural to ask about the persistence 
of the unexpectedness imposed by $X$. For example, in \cite[Corollary 2.12]{HMNT}, it is 
shown that a nondegenerate curve $C \subset \PP^3$ of degree $d=\deg C$ admits 
an unexpected hypersurface of degree $t$ for multiplicity $t$ at a general point for all $t \geq d$. 
Thus fixing $0 = j = t-m$, and fixing $C$, we have the persistence of unexpectedness as long as $t \geq d$.

Many of the results in this paper give formulas for the sequences $AV_{X,j} (m)$. Leaving 
aside the issue of whether ${\adim} (X,t,m) > 0$, this sequence can be interpreted both 
as a measure of  unexpectedness (how much bigger is the actual dimension than what one 
would expect?) and as a measure of persistence (how long is $AV_{X,j}(m)$ positive?). 
The fact that these sequences are represented by simple formulas, as we will see in the coming sections, is a pleasant bonus.
\end{remark}


\section{Generic initial ideals and unexpectedness}\label{s. gin and unexp}
In this section we relate the study of unexpected hypersurfaces of a subscheme $X\subseteq \PP^n$ 
to the generic initial ideal of $I_X$ with respect to the lexicographic order. Then, we prove that the $AV_{X,j}$ 
sequence, up to a shift, is an $O$-sequence. As a consequence of this result we are able to ensure the 
non-existence of unexpected hypersurfaces in several cases.
 
Let $R = K[x_0,x_1,\ldots,x_n]$ be a standard graded polynomial ring. In this section we only 
require $K$ to be infinite.  We assume the monomials of $R$ are ordered by $>_{lex}$, the 
lexicographic monomial order which satisfies $x_0> x_1 > \cdots > x_n$.	
We recall that a set $M\subseteq R$ of monomials is a {\em lex-segment} 
if the monomials have the same degree and they satisfy the condition that whenever $u,v$ 
are monomials with $u\geq v$ and $v\in M$, then $u\in M$ \cite{V}.
It is convenient to also refer to a vector subspace $W\subseteq R$
as a lex-segment if $W$ is spanned by a lex-segment in the previous sense.
We also recall that a homogeneous ideal $I\subseteq R$ is a {\em lex-segment ideal} if 
for each degree $d$ the the homogeneous component $I_d$ of $I$ of degree $d$
is a lex segment \cite{V}; see also \cite{Hu}.

For a graded ideal $I\subseteq R$, we will denote by $\gin(I)$ the generic initial ideal of $I$ with respect
to the monomial order $>_{lex}$. For an introduction to generic initial ideals, see for instance \cite{Gr} and  Section 15.9 in \cite{E}.
The next lemma relates the actual and virtual dimensions of a scheme in terms of the generic initial ideal of its ideal. 


\begin{lemma}\label{l. adim vdim and Gin} Let $X\subseteq \PP^n$ be a subscheme.  For any non-negative integers $t$ and $m$, we have 
	\begin{itemize}
		\item[\em (i)] ${\adim} (X,t,m)=\dim [\gin(I_X)\cap I_Q^m]_t$, where $Q = (1,0,\dots,0)$.
		\item[\em (ii)] ${\vdim} (X,t,m)={\vdim} (\gin(I_X),t,m)$.
	\end{itemize}
\end{lemma}

\begin{proof}
(i) Let $X'$ be the image of $X$ under a general linear change of variables, so $\hbox{gin}(I_{X})=\hbox{in}(I_{X'})$
and so the point $Q$ is general for $X'$.	
Moreover, if $\mu_1,\mu_2$ are monomials of the same degree with $\mu_1\in[I_Q^m]_t$ and $\mu_1>_{lex}\mu_2$,
then $\mu_2\in I_Q^m$, hence $\hbox{in} (I_{X'}\cap I_Q^m)=\hbox{in} (I_{X'})\cap I_Q^m$. 
(To see this note that both sides of the equality are monomial ideals, and 
that $\hbox{in} (I_{X'}\cap I_Q^m)\subseteq \hbox{in} (I_{X'})\cap I_Q^m$ is clear.
So suppose that $w\in \hbox{in} (I_{X'})\cap I_Q^m$ is a monomial.  
Then there is a form $W\in I_{X'}$ with $w=\hbox{in}(W)$. 
It follows that the other terms of $W$ have lex order less than $w$, so
each is in $I_Q^m$, hence $W\in I_Q^m$, and we have $W\in I_{X'}\cap I_Q^m$
so $w\in \hbox{in} (I_{X'}\cap I_Q^m)$ giving $\hbox{in} (I_{X'}\cap I_Q^m)\supseteq \hbox{in} (I_{X'})\cap I_Q^m$.)
Thus we have
${\adim}(X,t,m)=\dim [I_X\cap I_P^m]_t 
			      = \dim [I_{X'}\cap I_Q^m]_t 
			      = \dim [\hbox{in}(I_{X'}\cap I_Q^m)]_t 
			      = \dim [\hbox{in}(I_{X'})\cap I_Q^m]_t 
			      = \dim [\hbox{gin}(I_{X})\cap I_Q^m]_t$.

(ii) This is a consequence of $h_X=h_{R/(\gin(I_X))}$.
\end{proof}


\begin{remark} \label{r: partial elimination ideal} 
	Let $I=\oplus_k I_k$ be a homogeneous ideal in $R$. 
Then we make the following definition (see \cite[Definition 6.1]{Gr}):
		\[
		\widetilde K_d(I) =\bigoplus_k[I\cap I_Q^{k-d}]_k,
		\]
where $Q = (1,0,\dots,0)$, so
	 $\widetilde K_d(I)$ is a graded module over $K[x_1,\ldots,x_n]$. 
	 In particular, $\widetilde K_0(I)$ is an ideal: it is obtained from $I$ by \textit{eliminating} the variable $x_0$. 
	 Geometrically, it corresponds to the linear space of the cones having vertex at $Q$. Then Lemma \ref{l. adim vdim and Gin} (i) can be rephrased as follows: 	
	 \[
	 \dim (I_X \cap I_P^m)_t=\dim \left[\widetilde K_{t-m}\left(\gin(I_X) \right)\right]_t.
	 \]
\end{remark}

\begin{example}\label{Fig1Fig2Example}
Assume $\hbox{\rm char}(K)=0$. 
Here, given only the Hilbert function $h_X$ of a set of points $X\subset\PP^2$,
we show how information about $\gin(I_X)$ relates to whether or not $X$ has
an unexpected curve.
So let $X\subseteq \PP^2$ be a set of 13 points with 
$h_X=(1, 3, 6, 10, 12, 13, 13, \ldots)$,
so $X$ lies on three independent quartics but no cubics (such examples exist, as we see later in this example). Then 
$\dim [I_X]_6=15$, so ${\vdim} (X,6,5)=0$, hence $X$ admits an unexpected curve 
of degree~$6$ if and only if  ${\adim}(X,6,5)>0$.  
From Lemma \ref{l. adim vdim and Gin}(i),  
${\adim}(X,6,5)>0$ if and only if $\dim [\gin(I_X)\cap I_Q^5]_6>0.$ In this case, the ideal defining the point $Q$ is $I_Q=(y,z)\subseteq K[x,y,z]$, so the monomials of degree 6 contained in $[I_Q^5]_6$ are exactly those
lexicographically less than or equal to $xy^5$.
But (in characteristic 0) generic initial ideals are (by \cite{CaS}) strongly stable 
(see \cite{AL} for the definition and properties of strong stability),
hence if any monomial of degree 6 less than or equal to $xy^5$ is in $\gin(I_X)$,
then $xy^5$ is in $\gin(I_X)$ too. 
Thus $X$ admits an unexpected curve of degree 6 if and only if $xy^5\in \gin(I_X)$.

Now we determine all the monomials in $[\gin(I_X)]_{\le 6}$, assuming that $X$ admits no unexpected curves in degrees strictly lower than 6. With this assumption and using the fact that unexpected curves have degree at least 4 
\cite{A, FGST}, from Lemma \ref{l. adim vdim and Gin}(i) we have 
\begin{itemize}
	\item  $[\gin(I_X)\cap I_Q^3]_4=[\gin(I_X)\cap I_Q^4]_5=(0)$, so if $x^ay^bz^c \in [\gin(I_X)]_{\le 6}$ then $a\ge 2$, and
\smallskip
	\item ${\edim}(X,4,2)={\adim}(X,4,2)=0$ (by Bertini's Theorem), hence $[ \gin(I_X)\cap I_Q^2]_4=(0)$, so if $x^ay^bz^c \in [\gin(I_X)]_{4}$ then $a\ge 3.$
\end{itemize}
Collecting this information, we get only one strongly stable ideal through degree 5 (with the given Hilbert function), that is 
$(x^4, x^3y, x^3z, 
x^2y^3, x^2y^2z)$. So, $[\gin(I_X)]_{\le 5}$ is a lex segment.

Taking generators up to degree 6 of the lex-segment ideal with Hilbert function $h_X$ we get
\[ L = (x^4, x^3y, x^3z, x^2y^3, x^2y^2z, x^2yz^3, x^2z^4), \]
thus $X$ has an unexpected curve of degree 6 with a general multiple point of multiplicity 5 if and only if the component $[\gin(I_X)]_6$ 
fails to be equal to the degree 6 component of the lex-segment ideal with the same Hilbert function.

As promised, we now show that sets $X$ do arise.
We first give an example of a set $X$ of 13 points with no unexpected curves of degree 6 or less.
We get 12 points of the points of $X$
as the complete intersection of a general cubic and quartic; add to this a general point $Q$
to obtain $X$. It is easy to see that $X$ has the Hilbert function as specified above,
and by direct computation with Macaulay2 \cite{GS}, we find that
$X$ has no unexpected curves of degree 6 or less. In this case, we find from Macaulay2 that
$\gin(I_X)=(x^4, x^3y, x^3z,
x^2y^3, x^2y^2z,
x^2yz^3, x^2z^4,
xy^6, xy^5z,
xy^4z^3,
xy^3z^5,
xy^2z^7,
xyz^9,
xz^{11},
y^{13})$, hence $xy^5$, as claimed, does not occur.

We now give two examples of an $X$ with the specified Hilbert function which do have unexpected sextics.
Returning to the specfied Hilbert function, we see that two monomials of degree 6 are needed in the minimal set of generators of $\gin(I_X)$. If $X$ admits an unexpected curve of degree 6 with a general multiple point of multiplicity 5, we have already noticed that $xy^5\in \gin I_X$.
Thus, the strongly stable property forces $[\gin(I_X)]_{\le 6}$ to be either
\[
\mathcal G_1 = (x^4, x^3y, x^3z, x^2y^3, x^2y^2z,x^2yz^3, xy^5)
\]
or 
\[
\mathcal G_2 = (x^4, x^3y, x^3z, x^2y^3, x^2y^2z, xy^5, xy^4z).
\]
By direct computation we see that a set of points $X$ with $[\gin(I_X)]_{\le 6}= \mathcal G_2$ would have ${\adim}(X,6,5)=2$,
but by \cite[Corollary 5.5]{CHMN} this would mean $X$ has an unexpected quintic, contrary to our assumption
that $X$ admits no unexpected curves in degrees strictly lower than 6.
Thus $[\gin(I_X)]_{\le 6}= \mathcal G_2$ cannot occur if $X$ is a reduced set of 13 points.

However, $[\gin(I_X)]_{\le 6}= \mathcal G_1$ can occur; we give two examples.
Specifically, the following two sets of points $X_1$ and $X_2$ have $h_{X_1}=h_{X_2}=h_X$ 
and $[\gin(I_{X_i})]_{\le 6}= \mathcal G_1$. The lines (see Figure \ref{f.line config X }) dual to the points $X_1$
give what \cite{DMO} refers to as a $(1,1)$ tic-tac-toe arrangement:
    
\[\begin{array}{ccc}
X_1&:=&\{(1,0,0), (0,1,0), (0,0,1), (1,1,0),(0,1,1),(1,0,1), (-1,1,0),(0,-1,1),(-1,0,1),\\
&&(1,1,1), (-1,1,1),(-1,1,-1),(1,1,-1)\}.\\
\end{array}
	\]
The ideal defining $X_1$ is
\[
I_{X_1}=(y^3z - yz^3, x^3z - xz^3, x^3y - xy^3).
\]

  \begin{figure}[!ht]
  	\centering
  	\begin{tikzpicture}[scale=0.7]
   \draw (-4,-1) -- (4,-1);
   \draw (-4,0) -- (4,0);	
   \draw (-4,1) -- (4,1);
   
   \draw (1,-4) -- (1,4);
   \draw (0,-4) -- (0,4);	
   \draw (-1,-4) -- (-1,4);  	

   \draw (-4,-3) -- (3,4);
   \draw (-4,-4) -- (4,4);
   \draw (-3,-4) -- (4,3);

   \draw (-4,3) -- (3,-4);
   \draw (-4,4) -- (4,-4);
   \draw (-3,4) -- (4,-3);
 
  	\end{tikzpicture}
  	\caption{A sketch of the line configuration dual to the points of $X_1$ from Example \ref{Fig1Fig2Example}. 
	(The line at infinity, corresponding to the point $(0,0,1)$,  is not shown).}
  	\label{f.line config X } 
  \end{figure}
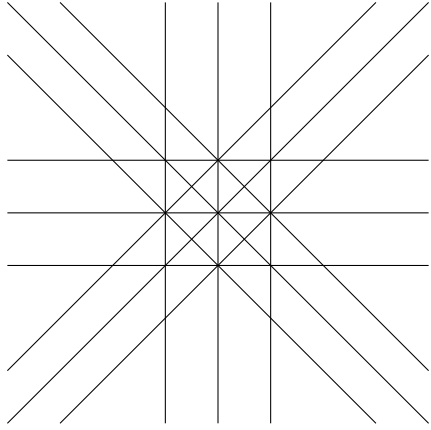

The second set of points is
\[\begin{array}{ccc}
X_2&:=&\{(1,0,0), (0,1,0), (0,0,1), (1,1,0),(0,1,1),(1,0,1), (-1,1,0),(0,-1,1),(-1,0,1),\\
&&(2,1,1), (-2,1,1),(-2,1,-1),(2,1,-1)\};\\
\end{array}
\]
it is defined by the ideal
\[I_{X_2}=(y^3z - yz^3, x^3z - xz^3 - 3xy^2z, x^3y - xy^3 - 3xyz^2).\]
(The lines dual to $X_2$ are shown in Figure \ref{f.line config X'}.)
 
  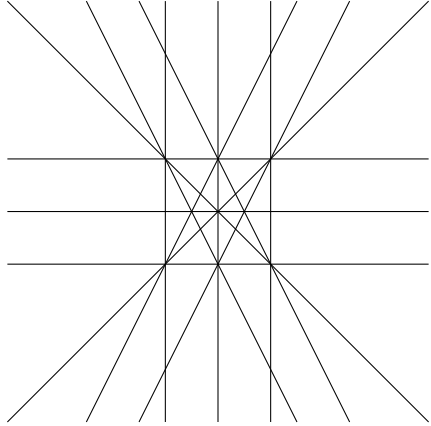
\begin{figure}[!ht]
  	\centering
  	\begin{tikzpicture}[scale=0.7]
   \draw (-4,-1) -- (4,-1);
   \draw (-4,0) -- (4,0);	
   \draw (-4,1) -- (4,1);
   
   \draw (1,-4) -- (1,4);
   \draw (0,-4) -- (0,4);	
   \draw (-1,-4) -- (-1,4);  	
  	
  	\draw (-2.5,-4) -- (1.5,4);
  	\draw (-4,-4) -- (4,4);
  	\draw (-1.5,-4) -- (2.5,4);
  	
  	\draw (-2.5,4) -- (1.5,-4);
  	\draw (-4,4) -- (4,-4);
  	\draw (-1.5,4) -- (2.5,-4);
  	
  	\end{tikzpicture}
  	\caption{A sketch of the line configuration dual to the points of $X_2$ from Example \ref{Fig1Fig2Example}. 
	(The line at infinity, corresponding to the point $(0,0,1)$, is not shown).}
  	\label{f.line config X'} 
  \end{figure}
\end{example}

We now show that the sequence $AV_{X,j}$ is the Hilbert 
function of a fixed standard graded algebra; since $AV_{X,j}(m)$ is defined for $m\ge 1$  we need to shift it by 1.

\begin{theorem}\label{t. AV is an O-sequence}
For any non-negative integer $j$, the sequence $AV_{X,j}$ 
shifted to the left by $1$ is an $O$-sequence. In particular, setting $J:=\gin(I_X)\ :\ x_0^{j+1},$  
the sequence $AV_{X,j}$ shifted to the left by $1$ coincides with the Hilbert function of $R/J$, i.e.,
	\[
	AV_{X,j}(d+1)=h_{R/J}(d), \ d\ge 0.
	\]

\end{theorem}
\begin{proof}Set $\mathfrak q= (x_1,\cdots, x_n)$ the ideal defining the point $Q=(1,0,\ldots,0)$.
For any $m\ge 1$ and for any non-negative integer $j$, from Lemma \ref{l.AV seq 1}, reasoning as in the proof of Lemma \ref{l. adim vdim and Gin}, we have
$$AV_{X,j}(m)=\dim\left[ \faktor{R}{(I_X + I_P^m)}\right]_{m+j}=\dim\left[  \faktor{R}{( \gin(I_X)  + \mathfrak q^m)}\right]_{m+j} .$$

For a monomial ideal $T$, it is easy to show that we can write $\left[T\right]_{m+j}$ as the following direct sum (and so the summands have only 0 in common):
\[ \left[T\right]_{m+j}= \left[x_0^{j+1}\cdot (T : x_0^{j+1}) \right]_{m+j}\bigoplus\left[ T\cap  \mathfrak q^m\right]_{m+j}.\]
In particular, 
$$\left[R\right]_{m+j}= \left[x_0^{j+1}\cdot (R : x_0^{j+1}) \right]_{m+j}\bigoplus\left[ R\cap  \mathfrak q^m\right]_{m+j}=\left[x_0^{j+1}R\right]_{m+j}\bigoplus\left[\mathfrak q^m\right]_{m+j},$$ 
so $\dim \left[{R/\mathfrak q^m}\right]_{m+j}=\dim \left[x_0^{j+1}R\right]_{m+j}$.
Similarly, 
\begin{align*}
\left[\gin(I_{X}) + \mathfrak q^m\right]_{m+j}&=\left[x_0^{j+1}\cdot ((\gin(I_{X}) + \mathfrak q^m) : x_0^{j+1}) \right]_{m+j}\bigoplus\left[(\gin(I_{X}) + \mathfrak q^m)\cap  \mathfrak q^m\right]_{m+j}\\
&=\left[x_0^{j+1}\cdot (\gin(I_{X}) : x_0^{j+1}) \right]_{m+j}\bigoplus\left[\mathfrak q^m\right]_{m+j},
\end{align*}
so $\dim \left[ {(\gin(I_{X}) + \mathfrak q^m)/\mathfrak q^m}\right]_{m+j}=\dim \left[x_0^{j+1}\cdot (\gin(I_{X}) : x_0^{j+1}) \right]_{m+j}$.
Thus
	\[\begin{array}{rcl}
	\dim\left[\dfrac{R}{\gin(I_{X}) +\mathfrak q^m}\right]_{m+j}&=&  \dim \left[\dfrac{R/\mathfrak q^m}{(\gin(I_{X}) + \mathfrak q^m)/\mathfrak q^m}\right]_{m+j}\\
	& & \\
	&=&  \dim \left[{R/\mathfrak q^m}\right]_{m+j} -\dim \left[ {(\gin(I_{X}) + \mathfrak q^m)/\mathfrak q^m}\right]_{m+j}\\
	& & \\
	&=&  \dim \left[x_0^{j+1}R \right]_{m+j} -\dim \left[ x_0^{j+1}\cdot (\gin(I_{X})  : x_0^{j+1})\right]_{m+j}\\
	& & \\
	&=&  \dim \left[R \right]_{m-1} -\dim \left[\gin(I_{X})  : x_0^{j+1}\right]_{m-1}\\
	& & \\
	&=&  \dim \left[R / (\gin(I_{X})  : x_0^{j+1})\right]_{m-1}.\\
	\end{array}
	\]	 
\end{proof}

As a consequence of Theorem \ref{t. AV is an O-sequence} we get a new criterion
(in this case, geometric) for 
the non-existence of unexpected hypersurfaces. 

\begin{corollary} \label{degenerate}
	If $X$ is a  reduced subscheme of $\PP^n$ contained in a hypersurface of degree $d+1\geq1$, 
	then for any $t\geq d+m$ and $m\ge 1$ we have 
\[
{\adim} (X,t,m) = {\vdim} (X,t,m).
\]
In particular, if $X$ is degenerate (meaning, $X$ is contained in a hyperplane, or, equivalently, $d=0$), then
$X$ admits no unexpected hypersurfaces for multiplicity $m$ in degrees $t\geq m$,
and hence no unexpected hypersurfaces of any kind.
\end{corollary}

\begin{proof}
	Since $I_X$ has an element of degree $d+1$, $x_0^{d+1}\in \gin(I_X)$. This implies  
	$1\in (\gin(I_{X})  : x_0^{t-m+1})$ when $t\geq d+m, m\geq 0$. 
	From Theorem \ref{t. AV is an O-sequence} and Definition \ref{def of AV},
	for $t=m+j, j\geq d, m\geq 1$, we have $0=AV_{X,j}(m)={\adim}(X,t,m) - {\vdim}(X,t,m)$.
	The last part, about unexpected hypersurfaces, follows from Definition \ref{unexpDef}.
\end{proof}

The next corollary compares the sequences $AV_{X,j} (m) $ and $  AV_{X,j+1}(m)$  for a subscheme $X$.

\begin{corollary} \label{ineq for AV}
		Let $X$ be a subscheme of $\PP^n$. Then
		\[AV_{X,j} (m) \ge  AV_{X,j+1}(m).\] 
\end{corollary}

\begin{proof}From Theorem \ref{t. AV is an O-sequence} we have, for any $i$,
\[
AV_{X,i} (m) = \dim \left[R / (\gin(I_{X})  : x_0^{i+1})\right]_{m-1}.
\]
So the statement is equivalent to proving that
\[
\dim \left[\gin(I_{X})  : x_0^{j+1} \right]_{m-1}\le \dim \left[\gin(I_{X})  : x_0^{j+2} \right]_{m-1},
\]
and this is trivial since we always have   $\gin(I_{X})  : x_0^{j+1} \subseteq \gin(I_{X})  : x_0^{j+2}.$ 
\end{proof}

\begin{lemma} \label{subscheme}
Let $P$ be a general point in $\PP^n$. 
Let $Y_1$ be  the zero-dimensional subscheme of $\PP^n$ defined by $I_P^m$ and let $Y_2$ be the subscheme of $Y_1$ defined by $I_P^{m-1}$.  Fix a positive integer~$t \geq m-1$ and consider the component $ [I_X]_t$, where $X$ is some subvariety of $\PP^n$. If $Y_1$ imposes $\binom{m-1+n}{n}$ independent conditions on $[I_X]_t$ then  $Y_2$ imposes $\binom{m-2+n}{n}$ independent conditions on $[I_X]_t$.
\end{lemma}

\begin{proof}
We know that  $\deg(Y_1) = \binom{m-1+n}{n}$ and $\deg(Y_2) = \binom{m-2+n}{n}$, and 
 that $Y_1$ and $Y_2$ impose independent conditions on $[R]_t$ since $t \geq m-1$ (by a regularity argument). In particular we know that 
\[
\dim [I_{Y_1}]_t = \dim [R]_t - \binom{m-1+n}{n} \ \ \hbox{ and } \ \ \dim [I_{Y_2}]_t = \dim [R]_t - \binom{m-2+n}{n}.
\]
Consider first the exact sequence
\[
0 \rightarrow I_{Y_1} \rightarrow I_{Y_2} \rightarrow A \rightarrow 0,
\]
where $A$ is the quotient, and is supported on $P$. We have
\[
\dim[A]_t = \binom{m-1+n}{n} - \binom{m-2+n}{n} = \binom{m-2+n}{n-1}.
\]
Let $\mathcal A$ be the sheafification of $A$. From the above exact sequence and the fact that $t \geq m-1$ (so $h^1(\mathcal I_P^m(t)) = h^1(\mathcal I_{Y_1}(t)) = 0$) and the fact that $I_{Y_1}$ and $I_{Y_2}$ are saturated, we get $h^0(\mathcal A(t)) = \dim [A]_t$. 

Remembering that $X$ and $P$ are disjoint, consider the exact sequence of sheaves
\[
0 \rightarrow \mathcal I_{X \cup Y_1}(t) \rightarrow \mathcal I_{X \cup Y_2}(t) \rightarrow \mathcal A(t) \rightarrow 0.
\]
Taking cohomology we obtain
\[
0 \rightarrow [I_X \cap I_P^m]_t \rightarrow [I_X \cap I_P^{m-1}]_t \rightarrow [A]_t \rightarrow \cdots
\]
Hence
\[
\begin{array}{rcl}
\dim [I_X \cap I_{Y_2}]_t & = & \dim [I_X \cap I_P^{m-1}]_t \\ \\
& \leq & \dim [I_X \cap I_P^m]_t + \dim [A]_t \\ \\
& = & \dim [I_X]_t - \binom{m-1+n}{n} + \binom{m-2+n}{n-1} \\ \\
& = & \dim [I_X]_t - \binom{m-2+n}{n}.
\end{array}
\]
But $Y_2$ cannot impose more conditions in degree $t$ than its degree, so we must have equality.
\end{proof}

Like Corollary \ref{degenerate}, the next result gives a criterion (again basically geometric), 
based on one single piece of information,
for a subvariety $X \subset \PP^n$ to admit no unexpected hypersurfaces of any degree or multiplicity at a general point.

\begin{proposition} \label{AV_X,0(alpha)}
Let $X \subset \PP^n$ be a subscheme. Let $\alpha:= \alpha(I_X)$ be the least degree $t$ such that $[I_X]_t\neq 0$.  If $AV_{X,0} (\alpha)=0$, then $X$ does not admit unexpected hypersurfaces, for any degree and multiplicity at a general point. 
\end{proposition}

\begin{proof}
By Remark \ref{equiv to unexp}, $X$ fails to admit an unexpected hypersurface 
of degree $t$ for multiplicity $m$ if and only if $T_{t,m}=\min({\adim}(X,t,m), AV_{X,t-m}(m))=0$.
Consider the table $T$ of values of $T_{t,m}$ for $t,m\geq1$. It suffices to show we always have
$T_{t,m}=0$. We will divide the table into three regions as follows:
\[
\begin{array}{cc|ccccc|ccccccccccccccccc}
& t & 1 & 2 & 3 & \dots & \alpha-1 & \alpha & \alpha +1 & \alpha+2 & \alpha+3 & \dots\\ 
m & &&&&&&&&& \\ \hline
1 & &&&&&&&&& \\ 
2 & &&&&&&&&& \\
\vdots & &&&&&&&  & {\rm\Large III}& \\
\alpha-1 & &&&&&&&&& \\ \cline{8-12}
\alpha & && && {\rm\Large I}&&&&& \\
\alpha +1 & &&&&&&&&& \\
\vdots & &&&&&&&&  {\rm\Large II} & \\

\end{array}
\]
As always, $P$ denotes a general point.

For $t < \alpha$ it is clear that $T_{t,m}={\adim}(X,t,m)=0$. 
This takes care of region I.

We now show that all entries in region II are 0. The condition $AV_{X,0}(\alpha) = 0$ gives us that $T_{\alpha, \alpha}=0$;
this the top left point of region II.
It also gives us that $Y_1$, defined by $I_P^\alpha$, imposes $\binom{\alpha-1+n}{n}$ independent conditions on $[I_X]_\alpha$.
But then by Corollary \ref{ineq for AV} we have $T_{j+\alpha,\alpha}=AV_{X,j}(\alpha) = 0$ for all $j \geq 0$ 
(this gives the top row of region~II). 

In the portion below the main diagonal of region II we have $m > t$, so $T_{t,m}={\adim}(X,t,m)$ is 0,
since there there can be no hypersurfaces with degree $t$ and multiplicity $m$ (unexpected or not). 
For the portion on and above the main diagonal of region II, 
we apply Theorem \ref{t. AV is an O-sequence}, which says that $AV_{X,j}(m)$ is an $O$-sequence. 
Recall that the top row of region II gives us $AV_{X,j}(\alpha) = 0$ for $j \geq 0$.
Then it follows for $i>0$ that $T_{\alpha+i+j,\alpha+i}=AV_{X,j}(\alpha+i) = 0$ as well, since a Hilbert function that attains a value 0 cannot subsequently become non-zero. Thus starting from each entry on the top row of region II (all of which are 0), the entries 
descending diagonally and to the right are all 0. Thus all entries in region II above the main diagonal are also 0.

We now show all entries in region III are 0.
Because $AV_{X,j}(\alpha) = 0$, the scheme $Y_1$ defined by the (saturated) ideal $I_P^\alpha$ 
imposes independent conditions on $[I_X]_t$ for all $t\geq\alpha$. Hence, applying Lemma \ref{subscheme}
iteratively, so does the scheme $Y_2$ defined by $I_P^k$ for $\alpha-1 \geq k \geq 1$.
Thus $T_{t,k}=AV_{X,t-k}(k) = 0$ for $t\geq\alpha$. Thus all entries in region III are 0, and we are done.
\end{proof}

The next proposition shows that if a subscheme $X$ admits an unexpected hypersurface of 
degree $t$ vanishing with multiplicity $m$ at a general point, then $[\gin(I_X)]_t$ is not a lex-segment. 
This gives a criterion (this time algebraic), again based on a single piece of information 
(admittedly more difficult to verify), for the non-existence of any unexpected hypersurfaces.

\begin{proposition}\label{p.lex segments are bad for unexpectedness}
	Let $X$ be a subscheme of $\PP^n$. Assume there exists $j\ge 0$ such that  
	$[\gin(I_X)]_{m+j}$ is a lex-segment and  $\hbox{adim(X,m+j,m)}>0.$ Then  
	 $AV_{X,j}(m)=0.$ In particular, if $\gin(I_X)$ is a lex-segment ideal, then $X$ does not admit any unexpected hypersurfaces.
\end{proposition}

\begin{proof}
     Let  $j\ge 0$ be an integer such that $X$ admits a hypersurface of degree $m+j$ 
     vanishing with multiplicity at least $m$ at a general point $P$. Then by Lemma \ref{l. adim vdim and Gin}(i) 
     we have ${\adim} (X,t,m)=\dim [\gin(I_X)\cap I_Q^m]_t>0$, so
     there is a monomial $x_0^sM\in \gin(I_X)$ where $M]\in I_P^m$ so has degree $\deg M=m+j-s\geq m$, 
     but $j\geq s\geq0$.
     Now, since $[\gin(I_X)]_{m+j}$ is a lex-segment, we have $x_0^sM \in \gin(I_X)$. 
     In order to prove the statement, it is enough by Theorem \ref{t. AV is an O-sequence} to show that
 $\dim [R/(\gin(I_X):x_0^{j+1})]_{m-1}=0$. So, take any monomial 
 $F\in R_{m-1}$. Then $x_0^{j+1}F\in[R]_{m+j}$ and $x_0^{j+1}F>_{lex} x_0^sM$. Thus, $x_0^{j+1}F\in \gin(I_X)$,
 so $F\in \gin(I_X):x_0^{j+1}$, hence $[R/(\gin(I_X):x_0^{j+1})]_{m-1}=0$.
\end{proof}

 It is natural to ask if the converse of Proposition \ref{p.lex segments are bad for 
	unexpectedness} is true, namely if it is true for a finite set of points $X$ that if 
$\gin(I_X)$ is {\em not} a lex-segment ideal then $X$ must admit some sort of 
unexpected hypersurface. Recalling that one must take with a grain of salt the 
generic initial ideal produced by a computer algebra program (is the change of 
variables ``general enough?"), a counterexample is given in Example \ref{ex.root A_n}. 
It would be interesting to have a theoretical procedure to determine if the generic initial 
ideal of a finite set of points is a lex-segment ideal or not. 

\begin{example}\label{ex.root A_n} 
Assume the characteristic of the field $K$ is $0$. Let $X_n\subseteq \PP^n$ be the set of $\binom{n+2}{2}$ points obtained from the root system $A_{n+1}$ as described in \cite{HMNT}, section 3.1. Specifically, $X_n$ consists of the $\binom{n+1}{2}$ points having one entry equal to 1, one equal to $-1$ and the rest 0, together with the $n+1$ coordinate points.
	 	
		The initial degree of $I_{X_n}\subseteq K[x_0, x_1, \ldots, x_n]$ is $\alpha(I_{X_n})=3$ for $n\ge 2$, because the product of any 3 indeterminates vanishes at $X_n$  and  $[I_{X_n}]_2= (0)$. To see that $X_n$ does not lie on a quadric we take $F:=\sum c_{ab}x_ax_b\in I_{X_n}$ and we show that $F=0$. Indeed, since $F$ vanishes at the coordinate points, $c_{ab}=0$ if $a= b$; also $c_{ab}=0$ for $a\neq b$ because $F$ vanishes at the point having no zero entries at the positions $a$ and $b$.    
		
		A computer calculation, by \cocoa\ \cite{cocoa}, showed that $AV_{X_n,0}(3)=0$ for $2\le n\le 12$.  
		Therefore, by Proposition \ref{AV_X,0(alpha)}, $X_n$ does not admit any unexpected hypersurface of any sort,  for $2\le n\le 12$.
		
		This result is consistent with \cite{HMNT}, where a computer search did not turn up any unexpected hypersurfaces
		for $X_n$ in the cases $2\leq n\leq 6$, $2 \leq d \leq 6$, $2 \leq m \leq d$.
		
		Furthermore, we checked with \cocoa\ that, for $4\le n\le 12$, the ideal defining the set of points $X_n\subseteq \PP^n$ has a generic initial ideal that is not a lex-segment ideal. In particular, in the cases $5 \le n\le 12$ we noticed that $\left[\gin(I_{X_n})\right]_3$ fails to be a lex-segment because $x_1x_n^2$ fails to belong to $\gin(I_{X_n})$.
		
		Interestingly, we have checked for $3\le n\le 12$ that the set of points $Y_n\subseteq \PP^n$,  constructed from $X_n$ by replacing in the coordinates of its points all the ``-1" with ``+1," admits an unexpected cone of degree 3, and $x_0x_n^2$ fails to belong to $\gin(I_{Y_n})$.
\end{example}


\section{On the sequence $AV_{X,1}(m)$ when $X$ is an irreducible ACM curve in $\PP^3$}\label{s.irreducible ACM curve in P3}

In this section we will assume that $K$ has characteristic zero. We have already seen in 
Theorem \ref{t. AV is an O-sequence} that for a subscheme $X \subset \PP^n$, the 
sequence $AV_{X,j}$ shifted to the left by 1 is an $O$-sequence. In this section we will abuse 
terminology and just say that the sequence is an $O$-sequence, often (but not always) suppressing 
the shift. Furthermore, if the positive part of the sequence $AV_{X,j}$ is finite, we will ignore the terms 
that are zero and just say that $AV_{X,j}$ is finite. In this section our focus is on the case $j=1$.

Recall that an SI-sequence is a finite, non-zero, symmetric $O$-sequence such that the first half is a 
differentiable $O$-sequence (i.e., also the first difference of the first half is an $O$-sequence). 
The significance of SI-sequences is that they characterize the $h$-vectors of  arithmetically 
Gorenstein subschemes of projective space whose artinian reductions have the Weak Lefschetz Property. 
In codimension three they characterize the $h$-vectors of all arithmetically Gorenstein subschemes. 
Note that SI-sequences are automatically unimodal. See \cite{Hrm} for properties of SI-sequences.

We have produced a great deal of experimental evidence for the following conjecture
concerning SI-sequences. We will shortly prove part of it.
We recall that an ACM scheme is always connected \cite[Theorem 18.12]{E}, hence a smooth
ACM scheme is irreducible.

\begin{conjecture}\label{SI conj}
Let $X \subset \PP^3$ be a smooth ACM curve not lying on a quadric surface.  
Then the sequence $AV_{X,1}$ is an SI-sequence (shifted by 1).   The last non-zero term in this sequence 
is $AV_{X,1}(\deg X - 5)$, so the SI-sequence ends in degree $\deg X - 6$.
\end{conjecture}

 \begin{remark}
 We point out two things. First, 
we have not yet seen a direct connection to Gorenstein algebras; we have only this numerical conjecture. It would be very interesting to tie Gorenstein algebras to the study of unexpected hypersurfaces in some way. Second, 
 sometimes the SI-sequences that we obtain have codimension three and sometimes codimension four.
 \end{remark}
 
 The case where $X$ {\em does} lie on a quadric surface is contained in Theorem \ref{part of conj}.
We first give examples to show that all of the other assumptions are needed in this conjecture. 

\begin{example} \label{hyp not hold}
The symmetry of the sequence $AV_{X,1}(m)$ requires all of the given assumptions. 
The following examples, which were run in either \cocoa\ \cite{cocoa} or Macaulay2 \cite{GS}, show that dropping a hypothesis can
result in the AV sequence either not being finite, or, if finite, not being symmetric.

\begin{itemize}

\item[(a)] ({\em $X$ satisfies all assumptions.}) Assume that  $X_1$ is a line in $\PP^3$ 
and $X$ is obtained as the residual to $X_1$ by two general cubic surfaces containing $X_1$. 
Then $X$ is a smooth ACM curve of degree 8 and genus 7, and  the positive part of the sequence 
$AV_{X,1}$ is $(1,2,1)$. This satisfies Conjecture \ref{SI conj}, and in particular it is symmetric
with its last value being in degree $8-6=2$.

\item[(b)] ({\em $X$ is not ACM}.) Assume that $X_1$ is the disjoint union of two lines in 
$\PP^3$ and $X$ is linked to $X_1$ by a general choice of two cubic surfaces. 
Then $X$ is a smooth, non-ACM curve of degree 7 and genus 4, and the positive part of  
the sequence $AV_{X,1}$ is $(1,2)$, which is not symmetric.

\item[(c)] ({\em $X$ is not in $\PP^3$}.) Assume that $X$ is a smooth surface of 
degree 8 in $\PP^4$ obtained from a plane by linking using two general hypersurfaces of 
degree 3. (The example in (a) is a hyperplane section of this one.) Then the positive part of  
the sequence $AV_{X,1}$ is $(1,3,4,4,\dots)$, which is not finite. (Note that its first difference is the sequence in (a).)

\item[(d)] ({\em $X$ is not equidimensional, but the curve part is ACM}.) Assume that $X$ is the 
residual in $\PP^3$ of a line inside the complete intersection of two cubics (as in 
(a)). Let $Y$ be the union of $X$ with a general point. Then the positive part of  $AV_{Y,1}$ is the sequence $(1,3,2)$,
which is not symmetric.

\item[(e)] ({\em $j = 0$ instead of $j=1$}.) Assume that $X$ is the curve in (a) but take $j=0$. 
Then the positive part of $AV_{X,0}$ is the sequence $(1,    4,    8,   11,   13,   14,   14,   14, \dots)$,
which is not finite.

\item[(f)] ({\em $j = 2$ instead of $j = 1$}.) Assume that $X$ is the curve in (a) but take $j=2$. 
Then $AV_{X,2}$ is the sequence $(0,0,\dots)$, and so is finite and vacuously symmetric, but does not end in the conjectured degree. However, linking the line 
$X_1$ from (a) using two surfaces of degree 4 gives a curve $X$ of degree 15 and genus 28, with the positive part of $AV_{X,2}$ being the sequence $(1,2,2)$,
which is not symmetric.

\item[(g)] ({\em $X$ is zero-dimensional}.) Assume that $X$ is the complete intersection in $\PP^3$ 
of three general quartic surfaces. Then the positive part of the sequence $AV_{X,1}$ is $(1,  4, 7, 8, 5)$
which is not symmetric.

\item[(h)] ({\em $X$ is ACM but not irreducible, 1}.) Let $Z \subset \PP^2$ be the set of 9 points 
coming from the $B_3$ root system (see \cite[Figure 2]{DIV} and \cite{HMNT}). Let $X$ be the cone over $Z$ with vertex 
at a general point $Q$. Then $X$ is an ACM union of lines (whose general hyperplane section admits an 
unexpected quartic curve), and the positive part of the sequence $AV_{X,1}$ is $(1,3,4,4,4,\dots)$, which is not finite.

\item[(i)] ({\em $X$ is ACM but not irreducible, 2}). Let $X$ be as in (h), except now let $Z \subset \PP^2$ be a general 
set of 9 points. The Hilbert function of $Z$ (and hence $X$) is the same as it was in (h), but now one computes that 
the positive part of the sequence $AV_{X,1}$ is $(1,3,3,3,\dots)$, which is not finite. (Thus the AV sequence depends on
geometry beyond the Hilbert function.)

\item[(j)] ({\em $X$ is ACM and irreducible but not smooth}). Let $Q$ be a general point and 
let $X$ be the complete intersection of two general quartic surfaces in $I_Q^3$. Then $X$ is 
ACM and we have verified that $X$ is irreducible (using Macaulay2), but the positive part of the 
sequence $AV_{X,1}$ is $(1, 4, 8, 12, 15, 16, 15, 12, 8, 4, 2, 2, 2, \dots)$, which is not finite. (Thus we need 
smoothness and not simply irreducibility in the statement of the conjecture.)

\end{itemize}
\end{example}

While we are not able to prove the full conjecture, we at least show  unimodality and the 
differentiability of the increasing part, using only irreducibility of $X$ and not necessarily 
smoothness. What is missing is the finiteness, the symmetry and the degree of the last 
positive term in the sequence when $X$ is smooth. Notice that nothing about our 
argument fails for (i) or (j) in Example \ref{hyp not hold}, so something more will be 
needed to prove the rest of the conjecture.

\begin{theorem} \label{part of conj} Let $X \subset \PP^3$ be an irreducible ACM curve. 

\begin{itemize}

\item[(a)] If $X$ lies on a quadric surface, then for each $j\geq1$ the sequence $AV_{X,j}(m)$ is zero.

\item[(b)] If $X$ does not lie on a quadric surface then the sequence $AV_{X,1}(m)$ is 
non-zero and unimodal. Furthermore, the increasing part is a differentiable $O$-sequence.

\end{itemize}
\end{theorem}

\begin{proof}	
(a) This follows by Corollary \ref{degenerate}.

(b) We assume that 
	
	\begin{itemize}
		\item $P \in \PP^3$ is a general point with defining ideal $I_P$, 
		
		\item $H$ is a general plane in $\PP^3$ containing $P$,
		
		\item $L {\color{blue} \in I_P}$ is a linear form defining $H$,
		
		\item $Z = X \cap H$.
		
	\end{itemize}
	
	\noindent 
	Note that $P$ may be taken to be a general point in $H$ with respect to the set $Z$.  
	In addition to the notation introduced in Remark~\ref{cohom interp}, recall from Notation \ref{Pm vs mP} that
	\begin{quotation}
		{\em we will denote the scheme defined by $I_{P|H}^m$ by $mP$, a fat point in the 
		plane, to distinguish it from the fat point scheme $P^m$ in $\PP^3$ defined by $I_{P}^m$.} 
	\end{quotation}
	Notice that 
	\[
	I_{X \cup P^{m+1}} : L = I_{X \cup P^m} \ \ \ \hbox{ and } \ \ \ (I_{X\cup P^m} + (L))^{sat} = I_{Z \cup mP}.
	\]
	Since $Z = X \cap H$ is a general hyperplane section of $X$, and $X$ is irreducible, $Z$ is 
	a set of points in linearly general position in $H$ (in this case meaning no three points of $Z$ 
	are collinear). Hence by \cite{CHMN} Corollary 6.8, $Z$ does not admit any unexpected 
	curves in the plane. Considering the exact sequence of sheaves 
	\[
	0 \rightarrow \mathcal I_{Z \cup (m+1)P} (m+2) \rightarrow \mathcal I_Z (m+2) \stackrel{r_{m+2}}{\longrightarrow} \mathcal O_{(m+1)P} (m+2) \rightarrow 0
	\]
	from Remark \ref{cohom interp} (where the ideal sheaves are on $H = \PP^2$), the fact that $Z$ 
	does not admit any unexpected curves means that $r_{m+2}$ has maximal rank on 
	global sections. (For some values of $m$ it will be injective, and eventually it will be surjective.)
	
	Now consider the commutative diagram of sheaves (the rows are exact since $P\not\in X$ and the columns are exact since $X$ is irreducible):
	\[
	\begin{array}{cccccccccccccccc}
	&& 0 && 0 && 0 \\
	&& \downarrow && \downarrow && \downarrow \\
	0 & \rightarrow & \mathcal I_{X \cup P^m}(m+1) & \rightarrow & \mathcal I_X (m+1) & \rightarrow & \mathcal O_{P^m}(m+1) & \rightarrow & 0 \\
	&& \phantom{{ \times L}} {\Big \downarrow} {{ \times L}} && \phantom{{ \times L}} {\Big \downarrow} {{ \times L}} && \phantom{{ \times L}} {\Big \downarrow} {{ \times L}} \\
	0 & \rightarrow & \mathcal I_{X \cup P^{m+1}}(m+2) & \rightarrow & \mathcal I_X (m+2) & \rightarrow & \mathcal O_{P^{m+1}}(m+2) & \rightarrow & 0 \\
	&& \downarrow && \downarrow && \downarrow \\
	0 & \rightarrow & \mathcal I_{Z \cup (m+1)P}(m+2) & \rightarrow & \mathcal I_Z(m+2) & \rightarrow & \mathcal O_{(m+1)P}(m+2) & \rightarrow & \phantom{.} 0 .\\
	&& \downarrow && \downarrow && \downarrow \\
	&& 0 && 0 && 0 
	\end{array}
	\]
	In cohomology we obtain
	{\footnotesize
		\[
		\begin{array}{cccccccccccccccc}
		&& 0 && 0 && 0 \\
		&& \downarrow && \downarrow && \downarrow \\
		0 & \rightarrow & [ I_{X \cup P^m}]_{m+1} & \rightarrow & [ I_X ]_{m+1}& \rightarrow & H^0(\mathcal O_{P^m}(m+1)) & \rightarrow & H^1(\mathcal I_{X \cup P^m}(m+1)) & \rightarrow & 0 \\
		&& {\big \downarrow} && {\big \downarrow} && {\big \downarrow} && \phantom{\alpha_{m+1}} {\big \downarrow} \alpha_{m+1} \\
		0 & \rightarrow & [ I_{X \cup P^{m+1}}]_{m+2} & \rightarrow & [ I_X ]_{m+2}& \rightarrow & H^0(\mathcal O_{P^{m+1}}(m+2)) & \rightarrow & H^1(\mathcal I_{X \cup P^{m+1}}(m+2)) & \rightarrow & 0 \\
		&& \downarrow && \downarrow* && \downarrow && \downarrow \\
		0 & \rightarrow & [ I_{Z \cup (m+1)P}]_{m+2} & \rightarrow & [ I_Z ]_{m+2}& \stackrel{r_{m+2}}{\longrightarrow} & H^0(\mathcal O_{(m+1)P}(m+2)) & \rightarrow & 
		\coker  (r_{m+2}) & \rightarrow & \phantom{.} 0  \\
		&&&& \downarrow && \downarrow && \downarrow \\
		&&&& 0 && 0 && 0
		\end{array}  
		\]  }
using the fact that $X$ is ACM for the first vertical surjection (marked by an asterisk).
		
To begin, we focus only on the first line of the above commutative diagram. For $m=0$ we obtain
\[
h^1(\mathcal I_{X\cup P^0} (1)) = h^1(\mathcal I_{X} (1)) = 0
\]
since $X$ is ACM. For $m=1$ we obtain
\[
h^1(\mathcal I_{X \cup P}(2)) = 
\left \{
\begin{array}{ll}
0 & \hbox{if $\dim [I_X]_2 \geq 1$} \\
1 & \hbox{if $\dim [I_X]_2 = 0$.}
\end{array}
\right.
\]
Since, by Remark \ref{cohom interp} we have $h^1(\mathcal I_{X \cup P^m}(m+1)) = AV_{X,1}(m)$, 
and we know the latter is an $O$-sequence shifted by one thanks to Theorem \ref{t. AV is an O-sequence}, 
this gives that the start of the shifted $O$-sequence (the value 1) corresponds to $h^1(\mathcal I_{X \cup P}(2))$. 
Thus the sequence is non-zero since $X$ does not lie on a quadric surface, so we have the first part of (b).
(If we had merely assumed that $X$ is not degenerate, this also shows that the sequence is non-zero if and only if $X$ 
does not lie on a quadric surface.)

So for the rest of the proof we assume that $X$ does not lie on a  quadric surface.
Applying the Snake Lemma to the above commutative diagram, the fact that $r_{m+2}$ has 
maximal rank means that also $\alpha_{m+1}$ has maximal rank (in the right-hand column of 
the  diagram). Then applying Remark \ref{cohom interp} as $m$ varies, since $r_{m+2}$ is initially 
injective and then surjective, the same is true for the map $\alpha_{m+1}$.  This shows that the 
sequence $AV_{X,1}(m)$ is  unimodal.

Finally, we argue that the increasing part of the sequence $\{h^1(\mathcal I_{X \cup P^m}(m+1))\}$ is a 
differentiable $O$-sequence (shifted by 2).

Thanks to  Theorem \ref{t. AV is an O-sequence} and Remark \ref{cohom interp}, we know 
that the sequence $\{ h^1(\mathcal I_{X \cup P^m}(m+1)) \}$ is an $O$-sequence, and we have 
just seen that the map $\alpha_{m+1}$ is injective as long as $\coker (r_{m+2})$ is non-zero. So 
we have only to show that $\dim ( \coker (r_{m+2}))$ is an $O$-sequence (shifted by~2). But 
because $r_{m+2}$ has maximal rank, when $\coker (r_{m+2})$ is non-zero we have from the 
exactness of the bottom row of the above diagram that
\[
\begin{array}{rcl} 
\displaystyle \dim [ \coker (r_{m+2}) ]_{m+2} & = & \displaystyle  \binom{(m+1)-1+2}{2} - \dim [I_Z]_{m+2} \\ \\
& = & h_{R/I_Z}(m+2) -  [ 2(m+2) + 1 ] 
\end{array}
\]
as long as this number is positive. 

So for $t \geq 2$ {\em we want to show that 
\[
k_{t-2} := h_{R/I_Z}(t) - (2t+1) 
\]
is an $O$-sequence as long as it is positive}. Let $s = t-2$, so we want to show that 
$k_s$ is an $O$-sequence for $s \geq 0$.
Since $X$ does not lie on a quadric surface, $Z$ does not lie on a conic (since $X$ is ACM). 
Thus $k_0 = \dim [ \coker (r_2) ]_2 = 6-5 = 1$. 

Now assume that $t \geq 3$, so $s \geq 1$. As long as $h_{R/I_Z}(t) = \binom{t+2}{2}$ 
(i.e., before $I_Z$ begins) this difference is $k_s = \binom{s+2}{2}$, which is an $O$-sequence. So 
assume $h_{R/I_Z}(t) < \binom{t+2}{2}$. Consider the $t$-binomial expansion of $h_{R/I_Z}(t)$. Since
\[
2t+1 < h_{R/I_Z}(t) < \binom{t+2}{2},
\]
we have 
\[
h_{R/I_Z}(t) = \binom{t+1}{t} + \binom{t}{t-1} + (\hbox{terms in degrees $\leq t-2$})
\]
and 
\[
(2t+1) = \binom{t+1}{t} + \binom{t}{t-1}.
\]
Hence $k_s = h_{R/I_Z}(t) - (2t+1)$ has an $s$-binomial expansion coming directly from the 
$t$-binomial expansion of $h_{R/I_Z}(t)$, by removing the first two binomial coefficients. 
But $h_{R/I_Z}$ is an $O$-sequence, so it obeys Macaulay's bound (Theorem \ref{macaulay thm}). 
Thus $k_s$ does as well, and so is an $O$-sequence.
\end{proof}


\section{Codimension 2 complete intersections in $\PP^n$} \label{ci section}

In this section we apply results of Section \ref{s. gin and unexp} to the case of complete 
intersections of codimension 2 in $\PP^n$. In particular, as introduced in 
Remark \ref{r: partial elimination ideal}, we get partial information on the actual dimension 
for a complete intersection of codimension 2 in a certain degree by using the theory of \textit{partial elimination ideals}. 
(See \cite{Gr} for background on partial elimination ideals and for Sylvester matrices, which appear
in the proof of Proposition \ref{p. codimension 2 complete intersection}.)
Complete intersections of codimension 2  have been investigated in several papers. Indeed, 
the following result can be deduced from Proposition 6.8 in \cite{Gr} and  Corollary 3.9 in \cite{CoS}.  
We include here a proof for completeness of the exposition and to show explicitly how the 
partial elimination theory affects the existence of  unexpected hypersurfaces.  
 
Recall that  $R = K[x_0,x_1,\ldots,x_n]$ denotes a standard graded polynomial ring and the monomials of $R$ are ordered by $>_{lex}$, the 
lexicographic monomial order which satisfies $x_0> x_1 > \cdots > x_n$.

\begin{proposition}\label{p. codimension 2 complete intersection}
	Let $C\subseteq \PP^n$ be a codimension 2 complete intersection. 
	Assume $C$ is defined by two, sufficiently general, forms of degree $a,b$ respectively. Say $a\le b$.
	Set, for any integer $0\le j< a$,
\[
\begin{array}{rcl}
t & := & (a-j)(b-j)+j \hbox{ and } \\ 
m& := & (a-j)(b-j).
\end{array}
\]
Then
	\[
	{\adim} (X,t,m)> 0.
	\]
\end{proposition}
\begin{proof}
	After a general change of coordinates, say
	\[
	\begin{array}{ccc}
	F &=& f_a+f_{a-1}x_0+f_{a-2}x_0^2+\cdots +f_{1}x_0^{a-1}+ f_0  x_0^a\\
	G &=& g_b+g_{b-1}x_0+g_{b-1}x_0^2+\cdots+ g_1x_0^{b-1}+ g_0 x_0^b\\
	\end{array}
	\]
	and define $J:=(F,G)$ to be the ideal generated by $F$ and $G$.
	Let $M$ be the following matrix of size $(a+b-2j)\times (a+b-2j)$: 
	{\tiny\[M:=\left(\begin{array}{ccccccccccccc}
		f_{a-j} & f_{a-j-1} & \cdots & f_1 & f_{0} & 0 &  \cdots  & 0 \\ 
		f_{a-j+1}   & f_{a-j} & \cdots & f_{2} & f_{1} & f_0 & \cdots  &  0\\ 
		\vdots & \vdots& \ddots \\
		f_{a} & f_{a-1} & \cdots &          &             &    & \\
		0 & f_a& \\
		\vdots& \vdots \\
		0 & 0  &   \cdots    &&&&\cdots & f_0\\
		g_{b-j} & g_{b-j-1} & \cdots &   &  &  &  \cdots &  0  \\ 
		g_{b-j+1}   & g_{b-j} & \cdots &  &  & & \cdots &  0\\ 
		\vdots & \vdots & \ddots\\
		g_{b } & g_{b-1} & \cdots &          &           &    & \\
		0 & g_b& \cdots \\
		\vdots & \vdots\\
		0 & 0 & \cdots      &&&& \cdots   & g_0\\\
		\end{array}\right).
		\]} 
(Note that the determinant of the matrix $M$ is one of the minors of the Sylvester matrix $Syl(F,G,x_0)$ of 
$F$ and $G$, and these minors belong to the partial elimination ideal $K_{j}(J)$. See in 
particular Corollary 3.9 in \cite{CoS} and Proposition 6.9 and the following Remark in \cite{Gr}.) 
One can check that $\det(M)$ is homogeneous and  it has degree $m$.
Let $M_i$ denote the cofactor of $M$ corresponding to $(-1)^{i+1}$ times the determinant of the matrix that results from deleting the $i$-th row and the first column of $M$.
	Consider the following $a+b-2j$ polynomials
	\[
	\begin{array}{cccccccrrcccccc}
	F & =& f_a &+& f_{a-1}x_0 &+& \cdots +& f_{a-j}x_0^j &+& f_{a-j-1}x_0^{j+1} &+& \cdots\\
	x_0F &=&     & & f_ax_0 & +&\cdots +& f_{a-j+1}x_0^j & +& f_{a-j}x_0^{j+1}  &+& \cdots\\ 
	\vdots   & &  & & &    & &  & & & \\
	x_0^{j-1}F &=&     & &  & & & f_{a}x_0^j & +& f_{a-1}x_0^{j+1} &+& \cdots\\ 
	\vdots   & &  & & &    & &  & & & \\
	x_0^{b-j-1}F &=&    \cdots   & &  & & & \\
	G &=& g_b &+& g_{b-1}x_0 &+& \cdots +& g_{b-j}x_0^j & +& g_{b-j-1}x_0^{j+1} &+& \cdots\\
	x_0G &=&      & & g_bx_0 & +&\cdots +& g_{b-j+1}x_0^j & +& g_{b-j}x_0^{j+1}  &+& \cdots\\ 
	\vdots   & &  & & &    & &  & & & \\
	x_0^{a-j-1}G &=&    \cdots   & &  & & & \\
	\end{array}
	\] 	
	Multiplying the above $a+b-2j$ polynomials respectively by $M_1, M_2, \ldots, M_{a+b-2j}$  and taking the sum of them, we get
	\[\begin{array}{ccl}
	T:=\left(\sum\limits_{i=1}^{b-j} M_ix_0^{i-1} \right) F&+& \left(\sum\limits_{i=1}^{a-j} M_{b-j+i}x_0^{i-1} \right) G=\\ &=& f_aM_1+g_bM_{b-j+1} +x_0 \left(\cdots \right)  +\cdots+ x_0^{j}(\cdots) +x_0^{j+1}(\cdots)+\cdots\\
	\end{array}
	\] 	
	Note that in the form $T\in K[x_1, \ldots x_n][x_0]$ the coefficient of $x_0^j$ is $\det(M)$. Also, note that all the coefficients 
	of the powers of $x_0$ greater than $j$ are zero; for instance the coefficient of $x_0^{j+1}$ 
	is the sum of the entries of the second column in $M$ multiplied by the cofactors of the entries 
	in the first column.  Of course $T$ belongs to the ideal generated by $F$ and $G$, therefore, since we performed a general change of variables, its leading term belongs to $\init(J)=\gin (I_C)$.  As noted above, the form $T$ can be written as $T=T_{m+j}+x_0 \left(T_{m+j-1} \right)  +\cdots+ x_0^{j}(T_m)$ where $T_i\in K[x_1, \ldots x_n]_i$ and $T_m=\det(M),$ thus
	$\init(T)=x_0^j\cdot \init(\det(M)) \in \gin(I_C)$. 
	Since $\init(T)$ vanishes with multiplicity $m$ at the point $Q:=(1,0,\ldots, 0)$, by Lemma \ref{l. adim vdim and Gin},  we are done.
\end{proof}

\begin{remark}\label{r. compute edim CI} As a consequence of Proposition 
\ref{p. codimension 2 complete intersection}, a general codimension 2 complete intesection 
$C\subseteq \PP^n$  admits an unexpected hypersurface of degree $t$ for multiplicity $m$
	whenever $\mbox{edim}\left(C,t,m \right)=0$, i.e. $\mbox{vdim}\left(C,t,m \right)\le 0$.
	In order to compute $\mbox{edim}\left(C,t,m \right)$, where $C$ is a complete 
	intersection of type $(a,b)$ in $\PP^n$,  take the short exact sequence
	\[
	0\to R(-a-b)\to R(-a)\oplus R(-b) \to I_{C} \to 0. 
	\]
	So, computing the dimension of the graded pieces of degree $t$, we get
	\[\dim [I_C]_t= {t-a +n\choose n}+{t-b +n\choose n}- {t-a-b +n\choose n}. \]
	Therefore, we have
	\[
	\mbox{edim}\left(C,t,m \right)=\max\left\{\dim [I_C]_t-{m-1+n \choose n},0\right\}.
	\]
\end{remark}

\begin{remark}The case $j=0$ is covered by \cite[Proposition 2.4]{HMNT}. The cone   
with vertex $P$ over a codimension 2 complete intesection $C\subseteq \PP^n$  
of type $(a,b)$ is an unexpected hypersurface for $C$ of degree $ab$ and multiplicity $ab$ at $P$. 
\end{remark}

If $j\ge 1$ it is enough to show that ${\vdim}\big(C, (a-j)(b-j)+j, (a-j)(b-j)\big)\le 0$ to 
ensure the existence of an unexpected hypersurface.
In the next proposition we deal with the case $j=1$ in $\PP^3$. Theorem \ref{part of conj} 
shows that a general complete intersection of type $(2,b)$ never admits an unexpected hypersurface 
with $j=1$. The following shows that when $a>2$ we do obtain unexpected hypersurfaces.

\begin{proposition} \label{unexpected CI}	Let $C\subseteq \PP^3$ be a general 
codimension 2 complete intersection defined by two forms of degree $a,b$ respectively. 
Say $a\le b$. 	Set\  $t:=(a-1)(b-1)+1$ and $m:=(a-1)(b-1)$.
	If $a >2$ then  $C$ admits an unexpected hypersurface of degree $t$ for multiplicity $m$. 
\end{proposition}
\begin{proof}

From Remark \ref{r. compute edim CI} we have
	\[\mbox{vdim}\left(C,t,m \right)={t-a +3\choose 3}+{t-b +3\choose 3}- {t-a-b +3\choose 3}-{t-2+3 \choose 3}.
	\]
	If $a=b=3$ then $t=5$ and
	\[\mbox{vdim}\left(C,t,m \right)={5\choose 3}+{5\choose 3}-{6 \choose 3}\le 0.
	\]
	A similar computation follows if $a=3$ and $b=4$. Note that, in this case $t+3-a-b=3$. 
	Since the integer $t+3-a-b$ increases with $a$ and $b$, in particular if $(a,b)>(3,4)$ 
	we have $t+3-a-b > 3$, so the binomial coefficient ${t +3-a-b\choose 3}$ is not zero. Then, 
	assuming $(a,b)>(3,4)$, after a standard computation we get
	\[
	\begin{array}{rl}
	\mbox{vdim}\left(C,t,m \right)& = -\dfrac{1}{2}a^2b-\dfrac{1}{2}ab^2+a^2+b^2+4ab-6a-6b+9=\\
	&\\
	&=-\dfrac{1}{2}(a-2)(b-2)(a+b-4)+1\le 0.\\
	\end{array}
	\]
	Thus $\mbox{vdim}\left(C,t,m \right)\le 0$.
\end{proof}

\begin{remark} \label{ans question}

Question 2.11 of \cite{HMNT} asks the following: Let $Z$ be a non-degenerate set of points in linear general position in $\PP^n$, $n \geq 3$. Is it true that there does not exist an unexpected hypersurface of any degree $t$ and multiplicity $m = t - 1$ at a general point? (All other possible combinations of $(n,t,m)$ are settled.)

The work in this section allows us to give a negative answer to this question. Indeed, let $X$ be a general complete intersection of type $(3,3)$ in $\PP^n$ ($n \geq 3$). Then let $t = (3-1)(3-1)+1 = 5$ and $m = t-1 = 4$. From Remark \ref{r. compute edim CI} we have 
\[
\hbox{vdim } (X,t,m) = \binom{n+2}{n} + \binom{n+2}{n} - \binom{n-1}{n} - \binom{n+3}{n} \leq 0
\]
for all $n \geq 3$. On the other hand, Proposition \ref{p. codimension 2 complete intersection} gives (with $a=b=3$ and $j=1$) that $X$ does lie on a hypersurface of degree 5 and multiplicity 4 at a general point, so this hypersurface is unexpected. Now take a set $Z$ consisting of sufficiently many general points on $X$, so that $[I_X]_5 = [I_Z]_5$. Then $Z$ is in linear general position since $X$ is irreducible, and so $Z$ admits an unexpected hypersurface and gives a negative answer to the question.

We believe that the natural extension of Proposition \ref{unexpected CI} to $\PP^n$ is also true, but we do not have a proof.
\end{remark}

\begin{remark}
One could make the following objection to Remark \ref{ans question} as an answer to Question 2.11 of \cite{HMNT}. That is that the degree 5 component of $I_Z$ is the same as the degree 5 component of $I_X$, so the base locus of $[I_Z]_5$ is $X$ and the geometry is really only about $X$ and not about $Z$.

To respond to this, we make the following tweak. Take $n = 4$. We choose the same $X$ as above (now a surface in $\PP^4$), and we still take $t=5$, $m=4$. One checks that the Hilbert function of $R/I_X$ is
\[
1, 5, 15, 33, 60, 96, 141 ...
\]
and we still have 
\[
\hbox{vdim }(X,t,m) = \binom{4+2}{4} + \binom{4+2}{4} - \binom{4-1}{4} - \binom{4+3}{4} = -5.
\]
But $X$ contains a set $Z$ of 225 points giving a $(3,3,5,5)$ complete intersection, hence
\[
\dim [I_{Z}]_5 = 32 = \dim [I_X]_5 + 2,
\]
so $\hbox{vdim}(Z,t,m) = -3$, hence all quintic hypersurfaces containing $Z$ with a general point of multiplicity 4
are unexpected, and there is certainly at least one (this being the one that contains $X$).
There still remains an objection: computations show in this case that $X$ and $Z$ 
both have a unique unexpected quintic
for multiplicity 4, so the one for $Z$ is the same one that we already got for $X$.
Nevertheless, what we have gained is that the component of the ideal for $Z$ in degree 5 is no longer the same as that of $X$, and indeed the base locus now is finite while the base locus in the original example was $X$ itself.
\end{remark}

\section{Cones, unmixed curves and unions with finite sets of points in 
$\PP^3$} \label{s.Unmixed curves and unions with finite sets in P3}

In this section, for the most part we restrict our attention to the case of subvarieties of $\PP^3$. 
We recall that
\[
AV_{X,0} (t) = {\adim} (X,t,t) - {\vdim} (X,t,t).
\]
We will determine the $AV_{Z,0}$ sequence for the cases where $Z=C$ is an 
equidimensional curve in $\PP^3$ and where $Z=C\cup X$ is the union of a 
curve $C$ and a finite set of points $X$. In particular, since $j=0$, we are focusing on the case of unexpected {\em cones}.
In both cases the geometric information on $C$ provides a description of the {\em persistence} of  
unexpected hypersurfaces.  We recall the following result, which reflects the persistence of 
unexpectedness for cones for a non-degenerate curve in $\PP^3$.

\begin{theorem}[\cite{HMNT} Corollary 2.12] \label{HMNT cone}
Let $C \subset \PP^3$ be a reduced, equidimensional, non-degenerate curve 
of degree $d \geq 2$ ($C$ may be reducible, singular, and/or disconnected). Let 
$P \in \PP^3$ be a general point. Let $k \geq d$ be a positive integer. Then 
$C$ admits an unexpected hypersurface of degree $k$ with multiplicity $k$ at $P$. When $k = d$, this hypersurface is unique.
\end{theorem}

Our next result not only reproduces the persistence 
given by the result of Theorem~\ref{HMNT cone} but also gives a measure of the unexpectedness in each degree. 

\begin{theorem} \label{uct}
Let $C \subset \PP^3$ be a reduced, equidimensional curve of degree $e$ and arithmetic genus $g$. Then
\[
AV_{C,0}(t) = \binom{e-1}{2} - g
\]
for all $t \geq e$. Moreover, $AV_{C,0}(e) = 0$ if and only if $AV_{C,0}(t)=0$ for $t\geq e$ if and only if $C$ lies in a plane.
\end{theorem}

\begin{proof}
Let $P$ be a general point. We are considering hypersurfaces of degree $t$ with multiplicity $t$ at 
$P$, which are cones with vertex $P$ that contain $C$. Since the projection of $C$ from 
$P$ is contained in the hyperplane section of such a cone, we must have $t \geq e$.
We first compute the virtual dimension of $I_{C \cup P^t}$:
\[
\begin{array}{rcl}
\dim [I_C]_t - \binom{t+2}{3} & = & \binom{t+3}{3} - h_C(t) - \binom{t+2}{3} =\\
& = & \binom{t+2}{2} - [te - g + 1].
\end{array}
\] 
(The fact that for $t \geq e$ we have $h_C(t) = te - g + 1$ follows from the main 
result of \cite{GLP}.)
On the other hand, if $S$ is such a cone of degree $t$ and if $Q$ is a point of $C$, 
then the line joining $P$ and $Q$ meets $S$ with multiplicity at least $t+1$, so it 
must lie on $S$. That is, $S$ contains as a component the cone over $C$ with vertex 
$P$. Then the actual dimension of $I_{C \cup P^t}$ is the (vector space) dimension of the 
linear system of plane curves of degree $t-e$, i.e. it is $\binom{t-e+2}{2}$. Then we obtain
\[
AV_{C,0}(t) = \binom{e-1}{2} - g
\]
after a simple calculation. The last part follows since $g = \binom{e-1}{2}$ if and only if $C$ is a plane curve, \cite[Proposition 2.1, Claim 1]{HMNT}.
\end{proof}

Even if your interest is for the case when $Z$ is a finite set of points (which we consider in the next section),
situations involving curves (such as we are looking at in this section) 
sometimes force themselves into the picture in subtle ways,
as the next example shows.

\begin{example}
Let $X_1 , X_2 \subset \PP^3 $ be finite sets of points with $h$-vectors, respectively,
\[
(1,3,6,5,3,3,2) \ \ \ \hbox{ and } \ \ \ (1,3,6,6,3,3,2).
\]
In both cases, the two 3's constitute maximal growth, viewing the $h$-vector as a Hilbert function,
and force ``many" of the points to lie on a curve of degree 3 \cite{BGM}. For the sake of this 
example, let us assume that in both cases this curve is a twisted cubic, that it contains 18 points 
of each of  $X_1$ and $X_2$ (with $h$-vector $(1,3,3,3,3,3,2)$ in both cases),  and that the 
remaining 5 points of $X_1$ and the remaining 6 points of $X_2$ are chosen generically.  
Then in the first case there is a unique unexpected cone of degree 5 with vertex at a general 
point, while in the second case there is no unexpected cone of degree 5. 
We omit details here since we will study this 
kind of situation in the next section (see especially Example \ref{1366332}).
\end{example}

Our next result can again be viewed as a measure of unexpectedness for cones in 
each degree, and a statement about the persistence of unexpected cones.

\begin{theorem}\label{t.C U X}
Let $X \subset \PP^3$ be a finite set of points. Let $C$ be a reduced, 
equidimensional curve of degree $e$ and arithmetic genus $g$. Assume that $X$ 
is disjoint from $C$. 
Let $t$ be the smallest integer such that 
\begin{enumerate}
\item[(i)] $|X| < \binom{t+2}{2}$, and 
\item[(ii)] $X$ imposes independent conditions on forms of degree $t$.
\end{enumerate}
Then 
\[
AV_{X \cup C,0}(t+e) = AV_{X,0}(t) + \left [ \binom{e-1}{2} - g \right ].
\]
\end{theorem}

\begin{proof}

Let $P$ be a general point. Let $S_P$ be the cone over $C$ with vertex $P$; we 
know $\deg S_P = e$, $S_P$ is reduced, and $S_P$ has multiplicity $e$ at $P$. Let $T_P$ 
be a surface (unmixed) of degree $t+e$ containing $X \cup C$ and having multiplicity $t+e$ at 
$P$. ($T_P$ is a cone over a suitable plane curve.)  Note that $S_P$ is a component of $T_P$, 
and note that since $P$ is general, no point of $X$ lies on $S_P$. 
Write $U_P$ for the residual to $S_P$ in $T_P$. Note that $\deg U_P = t$, $U_P$ has 
multiplicity $t$ at $P$ and $U_P$ contains $X$. We also know that ${\adim} (C,e,e) = 1$ 
(Theorem \ref{HMNT cone}). These observations imply
\[
\dim [I_X \cap I_P^t]_t = \dim [I_{X \cup C} \cap I_P^{t+e}]_{t+e}.
\]
We also remark that since $X$ imposes independent conditions on forms of degree $t$ and 
$C$ imposes independent conditions on forms of degree $e$ (by \cite{GLP}), we can conclude 
\begin{equation} \label{indep cond}
\hbox{\em $X$  also imposes independent conditions on $[I_{C}]_{t+e}$.}
\end{equation}

By definition we have
\[
AV_{X,0}(t) = \dim [I_X \cap I_P^t]_t -  \left [ \dim [I_X]_t - \binom{t+2}{3} \right ].
\]
Now we compute (using the observation about independent conditions)
\small{
\[
\begin{array}{lcl}
AV_{X \cup C,0}(t+e) & = \displaystyle  & \displaystyle \dim [I_{X \cup C} \cap I_P^{t+e}]_{t+e} -  \left [  \dim [I_{X \cup C}]_{t+e} - \binom{t+e+2}{3} \right ] \\ \\
& = & \displaystyle \dim [I_X \cap I_P^t]_t - \left [ \dim [I_{X \cup C}]_{t+e} - \binom{t+e+2}{3} \right ] \\ \\
& = & \displaystyle 
AV_{X,0}(t) +  \left [ \dim [I_X]_t - \binom{t+2}{3} \right ] - \left [ \dim [I_{X \cup C}]_{t+e} - \binom{t+e+2}{3} \right ] \\ \\
& = & \displaystyle
AV_{X,0}(t) + \left [ \binom{t+3}{3} - |X|  - \binom{t+2}{3} \right ] - (\dim [ I_{C} ]_{t+e} - |X|) + \binom{t+e+2}{3} \\ \\
& = & \displaystyle 
AV_{X,0}(t) + \binom{t+2}{2} - \dim [I_C]_{t+e} + \binom{t+e+2}{3} \\ \\
& = & \displaystyle 
AV_{X,0}(t) + \binom{t+2}{2} - \left [ \binom{t+e+3}{3} - (e(t+e) - g + 1) \right ] + \binom{t+e+2}{3} \\ \\
& = & \displaystyle 
AV_{X,0}(t) + \binom{e-1}{2} - g
\end{array}
\] }
(the fourth line uses (\ref{indep cond}) and the last line comes after a routine calculation).
\end{proof}

\begin{corollary} \label{X U plane curve}
Let $X \subset \PP^3$ be a finite set of points. Let $C$ be a reduced plane curve in 
$\PP^3$ of degree $d$ disjoint from $X$. Then $X$ has an unexpected cone of degree 
$t$ if and only if $X \cup C$ has an unexpected cone of degree $t+d$. Furthermore, 
$AV_{X,0} (t) = 
AV_{X\cup C,0}(t+d)$ for all $t$.
\end{corollary}

\begin{proof}
The arithmetic genus of a plane curve of degree $e$ is $\binom{e-1}{2}$.
\end{proof}


\section{Finite sets of points in $\PP^3$} \label{s.Finite sets of points in P3}

In this section we assume that $K$ has characteristic zero. 
One of the original motivations for this paper was to determine if there are any Hilbert 
functions for non-degenerate sets of points that {\em force} the existence of unexpected hypersurfaces of some sort. 
(A consequence of 
Corollary \ref{degenerate} is that given a finite $O$-sequence $(1,a_1,\dots,a_r)$, 
one can trivially find a set of points in some projective space, with this $h$-vector, that does 
not admit any unexpected hypersurfaces. One simply produces a set $X$ in $\PP^n$ for $n > a_1$ 
having this $h$-vector. Then $X$ is degenerate, hence admits no unexpected hypersurfaces. 
Thus it is more interesting to consider non-degenerate sets.)

It now seems plausible that in a strict sense there are no such Hilbert functions. Indeed,
we make the following conjecture (but see Theorem \ref{force unexp} and Corollary \ref{LGP}, which show that such Hilbert functions
do arise when combined with a little geometric information).

\begin{conjecture} \label{conj about existence}
For every possible $h$-vector $(1,n,a_2,\dots,a_r)$  for a non-degenerate, finite set of 
points in $\PP^n$,  there is a set of points $X$ with that $h$-vector such that $X$ 
does not admit any unexpected hypersurfaces of any degree and multiplicity.
\end{conjecture}

In trying to prove Conjecture \ref{conj about existence} we made the following observations. 
Recall that a {\em distraction} is a construction that converts, in particular,  an artinian monomial 
ideal in $K[x_0,\dots,x_{n-1}]$ to the ideal of a reduced set of points in $\PP^n$.  It was 
introduced in \cite{hartshorne}. See also \cite{MN2} for related constructions and results.

One way of constructing a reduced set of points with a given $h$-vector is 
to start with the artinian lex-segment ideal with Hilbert function $h$ and perform a distraction to 
produce a set of points $X$. Experimentally, it seems that very often $\hbox{gin}(I_X)$ is a 
lex-segment ideal. If this were always the case, we would be done by Corollary \ref{recall gin}:

\begin{corollary} \label{recall gin}
Let $h= (1,n, a_2,\dots,a_r)$ be a finite $O$-sequence. Let $X \subset \PP^n$ be a 
set of points with this $h$-vector. If $\gin(I_X)$ is a lex-segment ideal in 
$R = K[x_0,x_1,\dots,x_n]$, then $X$ does not admit any unexpected hypersurfaces, 
for any degree and multiplicity.
\end{corollary}

\begin{proof}
This follows immediately from Proposition \ref{p.lex segments are bad for unexpectedness}.
\end{proof}

Unfortunately, we have verified that the $h$-vector {\tt (1,3,6,10,5,5,2)} results in a set of 
points for which $\gin(I_X)$ is not a lex-segment ideal, and in fact one can check on \cocoa\  
that it admits an unexpected hypersurface of degree 4 with multiplicity 3. However, we were able to 
confirm that the union in $\PP^3$ of 22 general points on a plane curve of degree 5 and 10 
general points in $\PP^3$ results in a set of points with the desired $h$-vector, which 
does not admit any unexpected hypersurfaces, so the conjecture is still true for this $h$-vector 
even though the distraction does not produce the desired set of points. This should be 
contrasted with the end of Example \ref{1366332}; in this case we verified that the 
distraction does produce a set of points whose gin is a lex-segment. Thus the conjecture remains open.

We recall a result from \cite{BGM}, modified to fit our context. For a set of points $X$ we 
denote by $\langle [I_X]_{\leq d} \rangle$ the ideal generated by the polynomials in $I_X$ of 
degree $\leq d$. Recall also that for a subscheme $Y$  of $\PP^n$ we denote by $h_Y (t)$ its Hilbert function.

\begin{proposition}[\cite{BGM} Theorem 3.6] \label{BGM thm}
Let $X \subset \PP^3$ be a reduced, finite set of points with $h$-vector
\[
(1,3,a_2,a_3,\dots,a_k,d,d, a_{k+3},\dots,a_r)
\]
where $d \leq k+1$. Then
\begin{itemize}

\item[\rm (a)] $\langle [I_X]_{\leq k+1} \rangle$ is the saturated ideal of a reduced curve, $V$, of 
degree $d$ (not necessarily unmixed). Also, $I_X$ has no minimal generators in degree $k+2$, 
so $\langle [I_X]_{\leq k+1} \rangle = \langle [I_X]_{\leq k+2} \rangle$.

\item[\rm (b)] Let $C$ be the unmixed, one-dimensional part of $V$. Let $X_1$ be the subset 
of $X$ on $C$ and let $X_2$ be the subset of $X$ not on $C$; note $X = X_1 \cup X_2$. 
Then $\langle [I_{X_1}]_{\leq k+1} \rangle = I_C$, and  $V = C \cup X_2$.

\item[\rm (c)] $h_{X_1} (t) = h_X(t) - |X_2|$ for all $t \geq k$.

\item[\rm (d)] 
\[
\Delta h_{X_1} (t) = 
\left \{
\begin{array}{ll}
\Delta h_C (t) & \hbox{for } t \leq k+2; \\
\Delta h_X (t) & \hbox{ for } t \geq k+1.
\end{array}
\right.
\]

\end{itemize}

\end{proposition}

From now on we focus on the special case $j=0$, i.e. when the degree of the unexpected 
hypersurface is equal to the multiplicity at the general point. The following example shows 
that knowing the $h$-vector of a finite set of points, and taking into account the fact that the 
base locus of some component of $I_X$ contains a curve, is not enough
to ensure that $X$ has an unexpected surface. 
We generally have to know something more about the curve.

\begin{example} \label{1366332}
Consider the $h$-vector
\[
(1,3,6,6,3,3,2).
\]
The values in degrees 4 and 5 force the existence of a cubic curve of some sort in the 
base locus of $[I_X]_4$ and of $[I_X]_5$ for any finite set $X$ with this $h$-vector. We 
will look at a few different kinds of cubic curves to see how the differences in the geometry 
of the curves gives different behavior with respect to unexpected hypersurfaces (specifically cones). 
Our goal is not to give an exhaustive list of possible sets of points 
with this $h$-vector, but rather to highlight a few to see how they differ.
So consider the following  sets of points sharing this $h$-vector. 

\begin{itemize}
\item Let $X$ consist of 18 points on a twisted cubic $C$ (note that the  
$h$-vector of these 18 points is (1,3,3,3,3,3,2)) plus 6 general points. 
Then we claim that $X$ admits no unexpected cone of any degree. 

Notice that 
\[
\dim [I_X ]_t  = \left \{
\begin{array}{ll}
4 & \hbox{if } t = 3; \\
16 & \hbox{if } t = 4; \\
34 & \hbox{if } t = 5; \\
\binom{t+3}{3} - 24 & \hbox{if } t \geq 6 .
\end{array}
\right.
\]
We first consider the case $t \leq 5$, and we notice ${\edim} (X,t,t) = 0$ in this case.
Since any cone of degree $\leq 5$ containing $X$ must also contain $C$, it also contains 
the cone over $C$ (which is a surface of degree 3) as a component. But the projection 
from $P$ of the 6 general points gives 6 general points in the plane, and there is no conic 
through 6 general points. Thus $\dim [I_{X \cup P^t}]_t = 0$ for $t \leq 5$ so ${\adim}(X,t,t) = {\edim}(X,t,t)=0$.

Now let $t \geq 6$. We have
\[
{\vdim} (X,t,t) = \binom{t+3}{3} - 24 - \binom{t+2}{3} = \binom{t+2}{2} - 24 > 0.
\]
Now, the projection of the points on $C$ gives a set of points with 
$h$-vector \linebreak $(1,2,3,3,3,3,3)$ so adding six general points gives a set with 
$h$-vector $(1,2,3,4,5,6,3)$, hence  the general projection imposes independent 
conditions on plane curves of degree $t$. Thus the vector space dimension of this 
linear system (hence the vector space dimension of the family of cones of degree $t$ with vertex $P$) is the expected one.

Notice that if the $h$-vector had been $(1,3,6,6,3,3,3)$ then the projection would 
have $h$-vector $(1,2,3,3,3,3,3,1)$ and the above argument would not work for $t=6$. 
Indeed, Theorem \ref{force unexp} gives an unexpected sextic cone.

\medskip

\item Let $C$ be a set of three disjoint lines.  Let $X$ consist of 6 points on one of the 
lines and 7   points on each of the remaining lines, chosen generally, together with 4 
general points in $\PP^3$. (The $h$-vector of the points on $C$ is (1,3,5,3,3,3,2), 
and the $h$-vector of $X$ is again $(1,3,6,6,3,3,2)$.) Then the expected dimension in 
degree 5 is 0 as before, but since there is a pencil of conics through four general points in 
the plane we obtain ${\adim}(X,5,5) = 2$, i.e. there is a pencil of unexpected cones of degree 5. 

\medskip

\item Let $C$ be a smooth plane cubic curve and let $X$ consist of 17 points on $C$ 
(with $h$-vector (1,2,3,3,3,3,2)) plus a set $X_1$ of 7 general points in $\PP^3$. 
One can check that $1 = AV_{X,0}(5) = AV_{C \cup X_1,0} (5) = AV_{X_1,0}(2)$ in 
accordance with Corollary \ref{X U plane curve} but that there is no unexpected hypersurface because ${\adim}(X_1,2,2) = 0$.

\medskip

\item Let $C$ be a smooth plane cubic curve in $\PP^3$, and let $\lambda_1$ 
and $\lambda_2$ be general lines in $\PP^3$. Let $X$ consist of  17 general 
points on $C$, plus a subset $X_1$ of four general points on $\lambda_1$ and three 
general points on $\lambda_2$. One can check that $X$ also has the $h$-vector 
$(1,3,6,6,3,3,2)$ so we expect no surface of degree 5 with a point of multiplicity 5 
at a general point $P$. However, the cone over $C \cup \lambda_1 \cup \lambda_2$ 
is such a surface. But notice that  the one-dimensional 
component of the base locus of $[I_X]_5$ is only the plane cubic.

\end{itemize}

\noindent Thus the $h$-vector $(1,3,6,6,3,3,2)$ may or may not force an unexpected 
cone, depending mostly, but not entirely, on the cubic curve that is forced by the $h$-vector. 

It is worth noting that the $h$-vector $(1,3,6,5,3,3,2)$  (analyzed as above) admits 
an unexpected cone even when the cubic curve is a twisted cubic, and $(1,3,6,6,3,3,2)$ 
admits an unexpected cone even when the cubic is a plane cubic.

\end{example}

As mentioned at the beginning of this section, we do not believe that any finite 
$O$-sequence forces the existence of unexpected hypersurfaces for non-degenerate 
sets of points. However, some sequences force the existence of a curve in the base 
locus of at least  some components of the ideal, and if this curve is not a plane curve 
then we {\em can} find finite $O$-sequences that  force unexpected hypersurfaces. 
This idea is elementary, but a bit technical. Thus we will first look at an example, to 
make the proof of Theorem \ref{force unexp} clearer. 
We will refer to the notation of Theorem \ref{force unexp} in this example.

\begin{example}
Consider sets of points $X$ with the $h$-vector
\[
(1,3,6,9,8,7,\underbrace{6,6,\dots,6}_\ell)
\]
where $\ell \geq 6$. In the notation of Theorem \ref{force unexp} we have $k = 5$, $d = 6$, $N = (6-4) + (9-5) + (8-6) + (7-6) = 9$ and $m = 3$. Because of the values of this $h$-vector, Proposition \ref{BGM thm}  applies. We get that for $6 \leq t \leq \ell+5$,  $(I_X)_{\leq t}$  has a 1-dimensional base locus, $C$, of degree 6, together with a finite set $X_2$, which imposes independent conditions on hypersurfaces of degree $\geq 5$. In fact
\[
(I_X)_{\leq t} = (I_{C \cup X_2})_{\leq t}
\]
for $6 \leq t \leq \ell +5$.
Furthermore, using Lemma \ref{bound X2} as in the proof of Theorem \ref{force unexp}, we see that $|X_2| \leq N = 9$. {\em Assume that $C$ is not a plane curve.}

Now we look in degree 11, which is in the range $6 \leq t \leq \ell+5$.
Applying Theorem \ref{t.C U X} we obtain
\[
AV_{X,0}(11) = AV_{X_2 \cup C,0}(11) = AV_{X_2,0} (5) + \left [ \binom{6-1}{2} - g \right ] > 0, 
\]
where $g$ is the arithmetic genus of $C$. The fact that this is positive follows since $C$ is not a plane curve. 

We remark  that this works because the degree 11 is such that the value of the $h$-vector is still $6$ in that degree. Beyond degree $k+\ell$ there is no longer a curve, and Theorem \ref{t.C U X} no longer applies.

Now let $P$ be a general point in $\PP^3$ and consider the projection from $P$ to a general $\PP^2$. The image of $X_2$ is thus a set of $\leq 9$ points in the plane, and as such it lies on a plane curve of degree $m=3$ and hence also a plane curve of degree $k = 5$. The cone over this curve is a surface of degree 5 containing $X_2$ with multiplicity 5 at $P$. Together with the cone over $C$ (which has degree 6), we have a surface of degree 11 having multiplicity 11 at $P$. This means $\hbox{adim} (X,11,11) > 0$. Since also $AV_{X,0}(11) >0$, $X$ admits an unexpected cone of degree 1.
\end{example}

\begin{lemma}[\cite{BGM} Lemma 3.1] \label{BGM lemma}  
Let $I \subset R$ be an ideal satisfying $h_{R/I}(t) = \binom{t+m}{t}$ and $h_{R/I}(t+1) = \binom{t+1+m}{t+1}$. 
Then $[I]_t$ is the degree $t$ component of the saturated ideal of an $m$-dimensional linear space in $\PP^n$ (and similarly for $[I]_{t+1}$).
\end{lemma}

\begin{lemma} \label{bound X2}
Let $C$ be a reduced, unmixed, non-degenerate curve in $\PP^3$ of degree $d$. 
Let $h_C(t)$ be its Hilbert function. Then $\Delta h_C(t)$ 
has a sharp lower bound as follows: If $d=2$ or 3 then the lower bound is (respectively) 
\[
\begin{array}{c|ccccccccccccccc}
\hbox{deg } t & 0 & 1 & 2 & 3 & 4 & 5 & \dots  \\ \hline
& 1 & 3 & 2 & 2 & 2 & 2 & \dots 
\end{array}
\ \ \ \hbox{ and } \ \ \ 
\begin{array}{c|ccccccccccccccc}
\hbox{deg } t & 0 & 1 & 2 & 3 & 4 & 5 & \dots  \\ \hline
& 1 & 3 & 3 & 3 & 3 & 3 & \dots 
\end{array}
\]
If $d \geq 4$ then the lower bound is
\[
\begin{array}{c|ccccccccccccccc}
\hbox{deg } t & 0 & 1 & 2 & 3 & 4 & \dots & d-3 & d-2 & d-1 & d & d & \dots \\ \hline
& 1 & 3 & 4 & 5 & 6 & \dots & d-1 & d & d & d & d & \dots 
\end{array}
\]
\end{lemma}

\begin{proof}
In all cases, since $C$ is non-degenerate we must have $\Delta h_C(1) = 3$.
If $d=2$ and $C$ is non-degenerate then $C$ must be a pair of disjoint lines, and the first 
given $\Delta h_C$ is its Hilbert function. (So this is precisely the Hilbert function and not a lower bound.) If $d=3$ and $\Delta h_C(t) \leq 2$ for any $t \geq 2$ then by Macaulay it can 
never grow to 3, which it must do since $\deg C = 3$. Thus the first two cases are done.

Notice that the given sequence is $\Delta h_C$ for the curve $C$ consisting of the union of a 
plane curve of degree $d-1$ and a line, meeting at one point. Thus this sequence occurs.

Let $I_C$ be the saturated ideal of $C$. Since $R/I_C$ has depth $\geq 1$, if $L$ is a 
general linear form then the first difference of $h_{C}(t)$ is the Hilbert function of $R/(I_C,L)$ 
and so is an $O$-sequence. We know $\Delta h_{C}(t) = d$ for $t \gg 0$. If $\Delta h_{C}(2) \leq 2$ then by 
Macaulay's theorem it can  never grow to $d$, so we must have $\Delta h_{C}(2) \geq 3$. 

Suppose $\Delta h_C(2) = 3$. In order to eventually reach $d$, by Macaulay's theorem 
we must have $\Delta h_C(t) = t+1$ for $2 \leq t \leq d-1$. Then by Lemma \ref{BGM lemma} 
(taking $m=1$),  $[(I_C + (L))/(L)]_t$ is the degree $t$ component of a line in $K[x,y,z]$. 
Thus since $I_C$ is saturated, $[I_C]_t$ is the degree $t$ component of a plane, i.e. $C$ 
is a plane curve of degree $d$. This is impossible since $C$ is non-degenerate.

The same argument applies for all degrees $3 \leq t \leq d-2$: we must have $\Delta h_C(t) \geq t+1$ 
in order to reach $d$, and if we have equality then $C$ must be a plane curve. Thus the stated Hilbert function is the smallest possible.
\end{proof}

 If Conjecture \ref{conj about existence} is true, the following kind of result is the best that one can 
 hope for, in terms of finding a Hilbert function that forces unexpectedness (but we do 
 not claim that this result is optimal in any way). It says that for a certain class of Hilbert 
 functions (which we define specifically via some numerical conditions) for which a curve 
 is forced in some component of the ideal because of maximal growth,  if you {\em assume} 
 that this curve is not a plane curve, then {\em any} set of points with this Hilbert function 
 must admit unexpected cones. Conjecture \ref{conj about existence} thus implies that if, however, you allow the curve 
 to be a plane curve then a set of points can be found for which there is no unexpected cone.  

\begin{theorem}
 \label{force unexp}
Let $X$ be a set of points in $\PP^3$ with $h$-vector
\[
(1,3,a_2, a_3, \dots, a_k, \underbrace{d, d , \dots, d}_\ell),
\]
where $k \geq 2$ and $a_k > d$.  Assume
\[
2 \leq d \leq \min \{ k+1, \ell \}.
\]
In case $d \geq 4$, let  
\[
b_i = 
\left \{
\begin{array}{lll}
a_i - (i+2) & \hbox{for $2 \leq i \leq d-2$} \\
a_i - d & \hbox{for $d-1 \leq i \leq k$}. \\
\end{array}
\right.
\]
If $d = 2$ or $d=3$, we replace $i+2$ by the bounds given in the first two parts of Lemma \ref{bound X2}.

Set 
\[
N = \sum_{i=2}^k b_i 
\]
and
\[
m = \min \left \{ i \ | \ \binom{i+2}{2} >~N \right \}. 
\]
We also assume  $m \leq k$.

Let $C$ be the equidimensional curve of degree $d$ guaranteed by Proposition~\ref{BGM thm}.  
If $C$ is not a plane curve then $X$ admits an unexpected cone of degree $d+k$.

\end{theorem}

\begin{proof}
Since  $d \leq k+1$, Proposition \ref{BGM thm} applies in degree $k$.
In particular, we get from Proposition \ref{BGM thm} (c) that $X_2$ 
imposes independent conditions on $[I_X]_s$ for any $s \geq k$, hence it also imposes 
independent conditions on the complete linear system of forms of degree $s$; we will use 
the case $s=k$. 

Now we look in degree $d+k$. By Theorem \ref{t.C U X} we have 
\[
AV_{C \cup X_2} (d+k) = AV_{X_2,0}(k) + \left [ \binom{d-1}{2} - g \right ] > 0 
\]
since $C$ is not a plane curve.

Note that $N$ is an upper bound for $|X_2|$, thanks to Lemma \ref{bound X2}. If we 
denote by $\pi_P$ the projection from a general point $P$ to a general plane, the assumption $m \leq k$
guarantees that $\pi_P(X_2)$ lies on a curve of degree $k$. This means that $X_2$ lies on a 
cone of degree $k$ with vertex at $P$. If $S_P$ is the cone over $C$ with vertex $P$, the 
union of these cones is a surface of degree $d+k$ with multiplicity $d+k$ at $P$. 
Thus ${\adim} (C \cup X_2,d+k,d+k) > 0$, so we have an unexpected cone of degree $d+k$ for 
$C \cup X_2$. 
But $k+1 \leq d+k \leq k+\ell$ so  $[I_X]_{d+k} = [I_{C \cup X_2}]_{d+k}$, 
so also $X$ admits an unexpected cone of degree $d+k$.
\end{proof}

\begin{example}
As mentioned above, the preceding result is not meant to be optimal. 
Consider for instance the $h$-vector $(1,3,6,5,3,3,3)$. We have $d=3$, $\ell = 3$, $N = 3 + 2 = 5$, $m = 2$, $k = 3$. 
The cubic curve $C$ guaranteed in the base locus of $[I_X]_t$ for $t = 4,5,6$ is either a twisted cubic, 
the union of a line and a conic (meeting in 0 or 1 points), or the union of three lines (meeting in a total of $<3$ points). 
Considering possible Hilbert functions of such curves, the given $h$-vector forces a set $X_2$ of at most $N = 5$ points off the curve (as a result of Lemma~\ref{bound X2}).

The theorem guarantees an unexpected cone of degree 6. Indeed, 
\[
\dim [I_X]_6 - \binom{6-1+3}{3} = 60 - 56 = 4
\]
and since the projection of $\leq 5$ points to $\PP^2$ lies on  at least a 5-dimensional vector space of plane cubics, the cones over these cubics (with vertex at the general point $P$) together with the cone over $C$ confirm the conclusion that there is an unexpected sextic. 

However, these projected points also lie on at least one conic, so there is a quadric cone containing $X_2$, and together with the cone over $C$ we get  an unexpected quintic cone (since $\dim [I_X]_5 - \binom{5-1+3}{3} = 0$), which is not covered by the theorem. 

If we had allowed $C$ to be a plane cubic curve, Lemma \ref{bound X2} would no longer hold: the lower bound  in this case would be given by the sequence $(1,2,3,3,3 ,\dots)$ so $|X_2|$ could also be 6.
\end{example}

The following result gives a geometric property for a set of points that is enough to find $h$-vectors that force unexpected cones.

\begin{corollary}
\label{LGP}
Let $X$ be a set of points in $\PP^3$ in linear general position, and assume 
that $X$ has $h$-vector given by the numerical conditions in Theorem \ref{force unexp}. Then $X$ admits an unexpected cone.
\end{corollary} 

\begin{proof}
The assumption of linear general position forces $C$ to be non-degenerate.
\end{proof}



\begin{thebibliography}{BDHHSS}

\bibitem[ABR]{cocoa} J.~Abbott, A.M.~Bigatti, and L.~ Robbiano,
CoCoA : a system for doing Computations in Commutative Algebra, 
available at {\tt http://cocoa.dima.unige.it}.

\bibitem[A]{A}
S.~Akesseh,
{\em Ideal Containments under Flat Extensions and Interpolation on Linear Systems in $\PP^2$},
PhD thesis, University of Nebraska, 2017.

\bibitem[AL]{AL}
D. Alberelli and P. Lella,
{\em Strongly stable ideals and Hilbert polynomials},
J. Software for Algebra and Geom. 9 (2019) 1--9.

\bibitem[AH]{AH}
J. Alexander and A. Hirschowitz,
{\em An asymptotic vanishing theorem for generic unions of multiple points},
Invent. Math. 140:2 (2000) 303--325.

\bibitem[BDSSS]{BDSSS}
T. Bauer, S. Di Rocco, D. Schmitz, T. Szemberg and J. Szpond,
{\em On the postulation of lines and a fat line},
J. Symb. Comp.,  91 (2019) 3--16.

\bibitem[BGM]{BGM}
A. Bigatti, A.V. Geramita and J. Migliore, {\em Geometric consequences of extremal behavior in a theorem of Macaulay}, Trans. Amer. Math. Soc. {\bf 346}, no. 1 (1994), 203--235.

\bibitem[BH]{BH}
 W. Bruns and J. Herzog,  ``Cohen-Macaulay rings," Cambridge Studies in Advanced Mathematics, 39. Cambridge University Press, Cambridge, 1993.

\bibitem[CCG]{CCG} E. Carlini, M.V. Catalisano and A.V. Geramita. {\em On the Hilbert function of lines union one non-reduced point}. Annali della Scuola Normale Superiore di Pisa. Classe di scienze. 2016;15(1):69-84.

\bibitem[CaS]{CaS}
G. Caviglia and E. Sbarra,
{\em Zero-generic initial ideals},
Manus. Math. 148 (2015) 507--520.

\bibitem[CM]{CM} 
L. Chiantini and J. Migliore, {\em Sets of points which project to complete intersections}, preprint 2019 (arXiv:1904.02047).

\bibitem[CoS]{CoS} A.\ Conca and J.\ Sidman, {\em Generic initial ideals of points and curves}, J. Symb. Comp. 40(3) (2005) 1023--38.

\bibitem[CHMN]{CHMN}
D.\ Cook, B.\ Harbourne, J.\ Migliore and U.\ Nagel, 
{\em Line arrangements and configurations of points with an unexpected geometric property}, 
Compositio Math. 154:10 (2018) 2150--2194. 

\bibitem[CH]{CH}
S.\ Cooper and B.\ Harbourne,
{\em Regina Lectures on Fat Points}, pp. 147--187, in: Connections Between Algebra, Combinatorics, and Geometry, S. Cooper and 
S. Sather-Wagstaff editors, vol. 76 of Springer proceedings in mathematics \& statistics, 2014.

\bibitem[DIV]{DIV} 
R.\ Di Gennaro, G.\ Ilardi and J.\ Vall\`es,  
{\em Singular hypersurfaces characterizing the Lefschetz properties}, 
J.\ London Math.\ Soc. (2) {\bf 89} (2014), no.\ 1, 194--212.

\bibitem[DMO]{DMO}
M.\ Di Marca, G.\ Malara, A.\ Oneto, {\em Unexpected curves arising from special line arrangements}, to appear in J. Alg. Comb. 
(arXiv:1804.02730).

\bibitem[DHRST]{DHRST}
M. Dumnicki, B. Harbourne, J. Ro\'e, T. Szemberg and H. Tutaj-Gasi\'nska,
{\em Unexpected surfaces singular on lines in $\PP^3$},
arXiv:1901.03725.

\bibitem[E]{E} D. Eisenbud, {\em An Introduction to Commutative Algebra with a View Towards Algebraic Geometry},  1995
Springer-Verlag, New York.

\bibitem[FGST]{FGST} 
\L.\ Farnik, F.\ Galuppi, L.\ Sodomaco and W.\ Trok,
{\em On the unique unexpected quartic in $\PP^2$},  
J.\ Algebr.\ Comb. (2019) on-line
(arXiv:1804.03590).

\bibitem[Gi]{Gi}
A.\ Gimigliano,  {\em On linear systems of plane curves}, Thesis, Queen's University, Kingston, 1987.

\bibitem[Gr]{Gr} M.\ Green, {\em Generic initial ideals}, In: Six Lectures on Commutative Algebra. Bellaterra, 1996. In: Progr.
	Math., vol. 166. Birkh\"auser, Basel, pp. 119--186.

\bibitem[GLP]{GLP}
L.\ Gruson, R.\ Lazarsfeld and C.\ Peskine, {\em On a theorem of Castelnuovo, and the equations defining space curves}, 
Invent.\ Math.\ 72 (1983), no. 3, 491--506. 

\bibitem[GS]{GS} D.R. Grayson and M.E. Stillman,  {\em Macaulay 2, a software system for research in algebraic geometry}, (2002).

\bibitem[H]{Harb} 
B.\ Harbourne, 
{\em The geometry of rational surfaces and Hilbert functions of points in the plane}, Proceedings of
the 1984 Vancouver Conference in Algebraic Geometry, CMS Conf. Proc. {\bf 6},  Amer. Math. Soc., Providence,
RI (1986),  95--111.

\bibitem[HaH]{HaHa}
K. Hanumanthu and B. Harbourne,
{\em Real and complex supersolvable line arrangements in the projective plane},
arXiv:1907.07712.

\bibitem[HMNT]{HMNT}
B. Harbourne, J. Migliore, U. Nagel and Z. Teitler, {\em Unexpected hypersurfaces and where to find them}, to appear in Michigan Math. J. (arXiv: 1805.10626).

\bibitem[HMT]{HMT}
B. Harbourne, J. Migliore and H. Tutaj-Gasi\'nska, {\em New constructions of unexpected hypersurfaces in $\PP^n$}, Rev.\ Mat.\ Compl. (2020) on-line (arXiv:1904.03251).

\bibitem[Hrm]{Hrm}
T. Harima, {\em Characterization of Hilbert functions of Gorenstein Artin algebras with the Weak Stanley Property}, Proc. Amer. Math. Soc. {\bf 123} (1995), 3631--3638.

\bibitem[Hart]{hartshorne}
R. Hartshorne, 
{\em Connectedness of the Hilbert scheme}, 
Math. Inst. des Hautes Etudes Sci, 29 (1966) 261--304. 

\bibitem[HaHi]{HaHi}
   R. Hartshorne and A.  Hirschowitz,
   {\em Droites en position g\'en\'erale dans l'espace projectif},
   Algebraic geometry (La R\'abida, 1981), 169--188, Lecture Notes in Math., 961, Springer-Verlag 1982.
   
\bibitem[Hi]{Hi}
A.\ Hirschowitz, {\em Une conjecture pour la cohomologie des diviseurs sur les surfaces rationelles g\'en\'eriques}, J.
Reine Angew. Math. {\bf 397} (1989), 208--213.

\bibitem[Hu]{Hu}
H.\ A.\ Hulett,
{\em Maximum betti numbers of homogeneous ideals with a given hilbert function},
Comm. Algebra, 21:7 (1993) 2335--2350.

\bibitem[LU]{LU}
A. Laface and L. Ugaglia,
{\em A conjecture on special linear
systems of $\PP^3$},
Rend. Sem. Mat. Univ. Pol. Torino - Vol. 63, 1 (2005).

\bibitem[Mac]{Mac}
F.S. Macaulay, 
{\em Some properties of enumeration in the theory of modular systems}, 
Proc. London Math. Soc. 26 (1927), 531--555.

\bibitem[Mi]{migbook} J. Migliore, ``Introduction to liaison theory and deficiency modules," Birkh\"auser, Progress in Mathematics vol. {\bf 165} (1998).

\bibitem[MN]{MN2}
J. Migliore and U. Nagel,
{\em Lifting monomial ideals},
Comm. Algebra 28:12 (2000) 5679--5701.

\bibitem[Se]{segre}
B.\ Segre, {\em Alcune questioni su insiemi finiti di punti in geometria algebrica},  Atti del Convegno Internazionale di Geometria Algebrica (Torino, 1961),  Rattero, Turin, 1962, 15--33. 

\bibitem[S1]{S1}
J. Szpond, {\em Unexpected curves and Togliatti-type surfaces}, 
Math.\ Nach. 293 (2020) 158--168 (arXiv:1810.06607).

\bibitem[S2]{S2}
J. Szpond, {\em  Unexpected hypersurfaces with multiple fat points},
arXiv:1812.04032.

\bibitem[Tr]{Tr} W. Trok, {\em Projective Duality, Unexpected Hypersurfaces and Logarithmic Derivations of Hyperplane Arrangements}, in preparation (2019).

\bibitem[V]{V}
G. Valla, {\em On the Betti numbers of perfect ideals},
Comp. Math. 91:3 (1994) 305-319.

\end{thebibliography}
\end{document}